\documentclass[reqno]{amsart}
\usepackage{graphicx}
\usepackage[dvips]{epsfig}
\usepackage{bm}
\numberwithin{figure}{section}

\newtheorem{theorem}{Theorem}[section]
\newtheorem{lemma}{Lemma}[section]
\newtheorem{proposition}[lemma]{Proposition}
\newtheorem{definition}[lemma]{Definition}
\newtheorem{problem}{{\bf{Problem}}}

\numberwithin{equation}{section}

\begin{document}

\title[Contact discontinuities in 2-D infinitely long nozzles]{Contact discontinuities for 2-D inviscid compressible flows in infinitely long nozzles}

\author{Myoungjean Bae}
\address{Department of Mathematics, POSTECH, 77 Cheongam-Ro, Nam-Gu, Pohang, Gyeongbuk, Korea 37673; Korea Institute for Advanced Study, 85 Hoegiro, Dongdaemun-gu, Seoul 02455, Republic of Korea}
\email{mjbae@postech.ac.kr}

\author{Hyangdong Park}
\address{Department of Mathematics, POSTECH, 77 Cheongam-Ro, Nam-Gu, Pohang, Gyeongbuk, Korea 37673}
\email{hyangdong.park@postech.ac.kr}



\date{\today}

\begin{abstract}
We prove the existence of a subsonic weak solution $({\bf u}, \rho, p)$ to steady Euler system in a two-dimensional infinitely long nozzle when prescribing the value of the entropy $(= \frac{p}{\rho^{\gamma}})$ at the entrance by a piecewise $C^2$ function with a discontinuity at a point. Due to the variable entropy condition with a discontinuity at the entrance, the corresponding solution has a nonzero vorticity and contains a contact discontinuity $x_2=g_D(x_1)$. We construct such a solution via Helmholtz decomposition.
The key step is to decompose the Rankine-Hugoniot conditions on the contact discontinuity via Helmholtz decomposition so that the compactness of approximated solutions can be achieved. Then we apply the method of iteration to obtain a piecewise smooth subsonic flow with a contact discontinuity and nonzero vorticity.
We also analyze the asymptotic behavior of the solution at far field. 
\end{abstract}
\maketitle

\tableofcontents

\section{Introduction}
In $\mathbb{R}^2$, the steady flow of inviscid compressible gas is governed by {\emph{the Euler system}}:
\begin{equation}\label{ES1}
\left\{
\begin{split}
&\mbox{div}(\rho{\bf u})=0,\\
&\mbox{div}(\rho{\bf u}\otimes{\bf u}+p{\mathbb I}_2)=0\quad({\mathbb I}_2: \mbox{$2\times 2$ identity matrix}),\\
&\mbox{div}\left(\rho\Bigl(E+\frac{p}{\rho}\Bigr){\bf u}\right)=0.
\end{split}
\right.
\end{equation}
In \eqref{ES1}, the functions $\rho=\rho({\rm x})$, ${\bf u}({\rm x})=(u_1, u_2)({\rm x})$, $p=p({\rm x})$, and $E=E({\rm x})$ represent the density, velocity, pressure, and the total energy density of the flow, respectively, at ${\rm x}=(x_1, x_2)\in \mathbb{R}^2$. In this paper, we consider an {\emph{ideal polytropic gas}} for which $E$ is given by
\begin{equation}
\label{definition-energy}
E=\frac 12|{\bf u}|^2+\frac{p}{(\gamma-1)\rho}
\end{equation}
for a constant $\gamma>1$, called the {\emph{adiabatic exponent}}. With the aid of \eqref{definition-energy}, the system \eqref{ES1} is closed, and can be rewritten as
\begin{equation}\label{E-S}
\left\{
\begin{split}
&\mbox{div}(\rho{\bf u})=0,\\
&\mbox{div}(\rho{\bf u}\otimes{\bf u}+p{\mathbb I}_2)=0,\\
&\mbox{div}(\rho{\bf u}B)=0,
\end{split}
\right.
\end{equation}
for the {\emph{Bernoulli invariant}} $B$ given by
\begin{equation}\label{Ber}
B=\frac{1}{2}|{\bf u}|^2+\frac{\gamma p}{(\gamma-1)\rho}=\frac{1}{2}|{\bf u}|^2+\frac{\gamma}{\gamma-1}S\rho^{\gamma-1}.
\end{equation}
Here, 
$S=p/\rho^{\gamma}$
denotes the entropy.

Let $\Omega\subset \mathbb{R}^2$ be an open and connected set. Suppose that a non-self-intersecting $C^1$-curve $\Gamma$ divides $\Omega$ into two disjoint open subsets $\Omega^{\pm}$ such that $\Omega=\Omega^-\cup\Gamma\cup \Omega^+$. 
Suppose that ${\bf U}=({\bf u},\rho,p)$ satisfies the following properties:
\begin{itemize}
\item[$(w_1)$] ${\bf U}\in [L^{\infty}_{\rm loc}(\Omega)\cap C^1_{\rm loc}(\Omega^{\pm})\cap C^0_{\rm loc}(\Omega^{\pm}\cup \Gamma)]^4$;
\item[$(w_2)$] For any $\xi\in C_0^{\infty}(\Omega)$ and $k=1,2$,
\begin{equation*}
\int_{\Omega}\rho{\bf u}\cdot \nabla\xi\,d{\rm x}=\int_{\Omega}(\rho u_k{\bf u}+p{\bf e}_k)\cdot \nabla\xi\, d{\rm x}=\int_{\Omega}\rho{\bf u}B\cdot \nabla\xi \,d{\rm x}=0.
\end{equation*}
Here, ${\bf e}_k$ is the unit vector in the $x_k$-direction.
\end{itemize}

By integration by parts, one can directly check that ${\bf U}$ satisfies the properties $(w_1)$ and $(w_2)$ if and only if
\begin{itemize}
\item[$(w_1^*)$] ${\bf U}$ satisfies the property $(w_1)$;
\item[$(w_2^*)$] ${\bf U}$ is a classical solution to \eqref{E-S} in $\Omega^{\pm}$, and satisfies the Rankine-Hugoniot conditions
\begin{eqnarray}
\label{RH1}&&[\rho{\bf u}\cdot{\bf n}]_{\Gamma}=[\rho{\bf u}\cdot{\bf n}B]_{\Gamma}=0,\\
\label{RH2}&&[\rho({\bf u}\cdot{\bf n}){\bf u}+p{\bf n}]_{\Gamma}={\bf 0},
\end{eqnarray}
for a unit normal vector field ${\bf n}$ on $\Gamma$, where $[F]_{\Gamma}$ is defined by 
$$[F({\rm x})]_{\Gamma}:=\left.F({\rm x})\right|_{\overline{\Omega^-}}-\left.F({\rm x})\right|_{\overline{\Omega^+}}\quad\mbox{for}\quad {\rm x}\in\Gamma.$$
\end{itemize}

Let ${\bm \tau}$ be a tangential vector field on $\Gamma$. Due to $[\rho{\bf u}\cdot{\bf n}]_{\Gamma}=0$ in \eqref{RH1}, the condition \eqref{RH2} can be rewritten as
\begin{equation}\label{RH4}
\rho({\bf u}\cdot{\bf n})[{\bf u}\cdot{\bm \tau}]_{\Gamma}=0,\quad [\rho({\bf u}\cdot{\bf n})^2+p]_{\Gamma}=0.
\end{equation}

Suppose that $\rho>0$ in $\Omega$. Then, the first condition in \eqref{RH4} holds if either ${\bf u}\cdot{\bf n}=0$ holds on $\Gamma$, or $[{\bf u}\cdot{\bm \tau}]_{\Gamma}=0$.

\begin{definition}
\label{definition-wsol}

We define ${\bf U}=({\bf u}, \rho, p)$ to be a weak solution to \eqref{E-S} in $\Omega$ with a {\emph{contact discontinuity $\Gamma$}} if the following properties hold:
\begin{itemize}
\item[(i)] $\Gamma$ is a non-self-intersecting $C^1$-curve dividing $\Omega$ into two open subsets $\Omega^{\pm}$ such that $\Omega=\Omega^+\cup\Gamma\cup \Omega^-$;

\item[(ii)] ${\bf U}$ satisfies $(w_1)$ and $(w_2)$, or equivalently 
$(w_1^*)$ and $(w_2^*)$;

\item[(iii)] $\rho>0$ in $\overline{\Omega}$;

\item[(iv)] $\left({\bf u}|_{\overline{\Omega^-}\cap \Gamma}-{\bf u}|_{\overline{\Omega^+}\cap \Gamma}\right)(\rm x)\neq {\bf 0}$ holds for all ${\rm x}\in \Gamma$;

\item[(v)] ${\bf u}\cdot{\bf n}|_{\overline{\Omega^-}\cap \Gamma}={\bf u}\cdot{\bf n}|_{\overline{\Omega^+}\cap \Gamma}=0$, where ${\bf n}$ is a unit normal vector field on $\Gamma$.

\end{itemize}
\end{definition}

One can directly check from \eqref{RH1} and \eqref{RH4} that ${\bf U}=({\bf u}, \rho, p)$ is a weak solution to \eqref{E-S} in $\Omega$ with a contact discontinuity $\Gamma$ if and only if the following properties hold:
\begin{itemize}
\item[$(i')$] The properties (i)-(iv) stated in Definition \ref{definition-wsol} hold;
\item[$(ii')$] $[p]_{\Gamma}=0$ and  ${\bf u}\cdot{\bf n}=0$ on $\Gamma$.
\end{itemize}

\begin{figure}[!h]
\centering
\includegraphics[width=0.65\columnwidth]{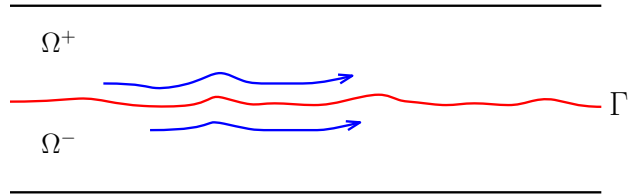}
\caption{Contact discontinuity}
\end{figure}

The goal of this paper is to prove the existence of subsonic weak solutions to \eqref{E-S} with contact discontinuities in the sense of Definition \ref{definition-wsol} in a two-dimensional infinitely long nozzle, and to analyze asymptotic behaviors of the contact discontinuities at far field. 
There are many studies of smooth subsonic solutions to Euler system, see \cite{C, Ch-D-X, D-D, D-X-X, D-X-Y,  D-Lu, D-L, X-X,  X-X-2, X-X-1} and references cited therein.
As far as we know, there are few results on the existence of solutions to Euler system with contact discontinuities \cite{B, Ch-H,Ch-K-Y1, Ch-K-Y2,W-Y}.
In \cite{B}, the existence of subsonic solutions with zero vorticity and contact discontinuities in two dimensional nozzles was established. 
In \cite{Ch-H}, the existence of subsonic solutions with large vorticity and characteristic discontinuities (vortex sheets or entropy waves) in two dimensional nozzles was proved.
In \cite{Ch-K-Y1,Ch-K-Y2}, transonic characteristic discontinuities in two dimensional nozzles were studied.
In \cite{W-Y}, supersonic contact discontinuities in three-dimensional isentropic steady flows were studied.

In this paper, we prove the existence of a subsonic weak solution $({\bf u}, \rho, p)$ to steady Euler system in a two-dimensional infinitely long nozzle when prescribing the value of the entropy $(=\frac{p}{\rho^{\gamma}})$ at the entrance by a piecewise $C^2$ function with a discontinuity at a point. Due to the variable entropy condition with a discontinuity at the entrance, the corresponding solution has a nonzero vorticity and contains a contact discontinuity $x_2=g_D(x_1)$. We construct such a solution via Helmholtz decomposition. By using Helmholtz decomposition, smooth subsonic flows for the full Euler-Poisson system with nonzero vorticity were studied in \cite{B-D-X,B-W}.
The challenge of this work, however, is to decompose the Rankine-Hugoniot conditions on the contact discontinuity via Helmholtz decomposition. Such a decomposition should be done carefully so that the compactness of approximated solutions can be achieved. To analyze the asymptotic behavior of the solution, we use the stream function formulation and energy estimates. 




The rest of the paper is organized as follows.
In Section \ref{sec-results}, we formulate the main problem of this paper, and state its solvability(Theorem \ref{3D-inf-Thm}(a)) and the asymptotic limit of the solution (Theorem \ref{3D-Far-Thm}(b)) as the main theorem.
In Section \ref{3D-Hel}, we use the method of Helmholtz decomposition to
reformulate the problem introduced in Section \ref{sec-results}, and state its solvability as Theorem \ref{3D-HD}. Then, we prove that Theorem \ref{3D-HD} implies Theorem \ref{3D-inf-Thm}(a) stated in  Section \ref{sec-results}. As we shall see later, the problems given in Sections  \ref{sec-results} and \ref{3D-Hel} are free boundary problems in an unbounded domain. To construct a solution to the free boundary problems in an unbounded domain, free boundary problems in cut-off domains will be formulated and solved  in  Section \ref{sec-cut}. 
Based on the results of Section \ref{sec-cut}, we prove Theorem \ref{3D-HD} from which Theorem \ref{3D-inf-Thm}(a) follows. Finally, the asymptotic behavior of the solution at far field is analyzed in Section \ref{sec-ex}.

\section{Main Theorems}\label{sec-results}
We define an infinite nozzle
\begin{equation}
\label{definition-N}
\mathcal{N}:=\left\{(x_1,x_2)\in\mathbb{R}^2:\mbox{ }x_1>0,\mbox{ }-1<{x_2}<1\right\}.
\end{equation}
The wall $\Gamma_{\rm w}$, upper wall $\Gamma_{\rm w}^+$, entrance $\Gamma_{\rm en}$, and a part of the entrance $\Gamma_{\rm en}^+$ of the nozzle $\mathcal{N}$ are defined as 
\begin{equation*}
\begin{split}
&\Gamma_{\rm w}:=\partial\mathcal{N}\cap\{x_2=1,-1\},\quad
\Gamma_{\rm w}^+:=\Gamma_{\rm w}\cap\{x_2=1\},\\
&\Gamma_{\rm en}:=\partial\mathcal{N}\cap\{x_1=0\},\quad
\Gamma_{\rm en}^+:=\Gamma_{\rm en}\cap\{0\le x_2\le 1\}.
\end{split}
\end{equation*}
We consider two layers of flow in $\mathcal{N}$ separated by the line $x_2=0$ with satisfying the following properties:
\begin{itemize}
\item[(i)] For fixed $\rho_0^{\pm}>0$ and $u_0>0$,  the velocity and density of top and bottom layers are given by  $(u_0,0)$, $\rho_0^+$ and $(0,0),$ $\rho_0^-$, respectively;
\item[(ii)] The pressure of both top and bottom layers is given by a constant $p_0>0$;
\item[(iii)] The flows in top and bottom layers are subsonic, i.e.,
$${u_0}<c_0\quad\mbox{for the sound speed}\quad c_0=\sqrt{\frac{\gamma p_0}{\rho_0^+}}.$$
\end{itemize}
Then a piecewise constant vector
\begin{equation*}
U_0(x_1,x_2):=\left\{
\begin{split}
(u_0,0,\rho_0^+,p_0)\quad&\mbox{for}\quad {x_2}>0,\\
(0,0,\rho_0^-,p_0)\quad&\mbox{for}\quad{x_2}<0,
\end{split}\right.
\end{equation*}
 is a weak solution of the Euler system \eqref{E-S} with a contact discontinuity on the line $x_2=0$.
In this case, the entropy $S_0$ and the Bernoulli function $B_0$ are piecewise constant functions with
\begin{equation}\label{def-B0}
\begin{split}
&S_0(x_1,x_2)=\left\{\begin{split}
	\frac{p_0}{(\rho_0^+)^{\gamma}}=:S_0^+\quad&\mbox{for}\quad x_2>0,\\
	\frac{p_0}{(\rho_0^-)^{\gamma}}=:S_0^-\quad&\mbox{for}\quad x_2<0,\end{split}\right.\\
& B_0(x_1,x_2)=\left\{\begin{split}
	\frac{1}{2}u_0^2+\frac{\gamma p_0}{(\gamma-1)\rho_0^+}=:B_0^+\quad&\mbox{for}\quad x_2>0,\\
	\frac{\gamma p_0}{(\gamma-1)\rho_0^-}=:B_0^-\quad&\mbox{for}\quad x_2<0.\end{split}\right.
\end{split}
\end{equation}

Our main concern is to solve the following problem.
\begin{problem}\label{3D-Prob1}
Fix $\epsilon\in(0,1/10)$ and $\alpha\in(0,1)$.
For given functions 
$S_{\rm en}=S_{\rm en}(x_2)$, $v_{\rm en}=v_{\rm en}(x_2)$ on $\Gamma_{\rm en}$, suppose that they satisfy
\begin{equation}
\label{en-epsilon}
\begin{split}
&(S_{\rm en},v_{\rm en})\equiv(S_0^-,0)\quad\mbox{on}\quad \Gamma_{\rm en}\backslash\Gamma_{\rm en}^+, \\
&v_{\rm en}\equiv0\quad\mbox{on}\quad \Gamma_{\rm en}^{\epsilon}:=\Gamma_{\rm en}^+\setminus\{\epsilon<x_2<1-\epsilon\},
\end{split}
\end{equation}
and 
\begin{equation*}\label{3D-rem11}
\sigma:=\|S_{\rm en}-S_0\|_{2,\alpha,\Gamma_{\rm en}^+}+\|v_{\rm en}\|_{1,\alpha,\Gamma_{\rm en}^+}\le \sigma_0
\end{equation*}
for a sufficiently small constant $\sigma_0>0$ to be specified later.

Find a weak solution $U=({\bf u},\rho,p)$ to \eqref{E-S} with a contact discontinuity 
\begin{equation*}
\Gamma_{g_D}:x_2=g_D(x_1)
\end{equation*} 
in the sense of Definition \ref{definition-wsol}  in $\mathcal{N}$  such that 
\begin{itemize}
\item[(a)]  $g_D(0)=0$.
\item[(b)] Subsonicity: $${|{\bf u}|}<c\quad\mbox{for the sound speed}\quad c=\sqrt{\frac{\gamma p}{\rho}}\quad\mbox{in}\quad\overline{\mathcal{N}}.$$
\item[(c)] Positivity of density: $\rho>0$ {in} $\overline{\mathcal{N}}.$
\item[(d)] At the entrance $\Gamma_{\rm en}$, $U$ satisfies the boundary conditions
\begin{equation*}\label{3D-BC-ent1}
\begin{split}
\frac{p}{\rho^{\gamma}}=S_{\rm en},\quad {\bf u}\cdot{\bf e}_2=v_{\rm en}\quad\mbox{on}\quad \Gamma_{\rm en}.
\end{split}
\end{equation*}
\item[(e)] On $\Gamma_{g_D}$, 
$U$ satisfies the Rankine-Hugoniot conditions, i.e.,
\begin{equation*}\label{3D-BC-Cont1}
 [p]_{\Gamma_{g_D}}=0,\quad{\bf u}\cdot {\bf n}_{g_D}=0\quad\mbox{on}\quad \Gamma_{g_D},
\end{equation*}
where ${\bf n}_{g_D}$ denotes a unit normal vector field on $\Gamma_{g_D}$.

\item[(f)] On the wall $\Gamma_{\rm w}$,
$U$ satisfies the slip boundary condition, i.e., 
$${\bf u}\cdot{\bf e}_2=0\quad\mbox{on}\quad\Gamma_{\rm w}.$$
\item[(g)] The Bernoulli function $B$ is a piecewise constant function,
\begin{equation*}
B(x_1,x_2)=\left\{\begin{split}
&B_0^+\quad\mbox{for}\quad x_2>g_D(x_1),\\
&B_0^-\quad\mbox{for}\quad x_2<g_D(x_1),
\end{split}\right.
\end{equation*}
where $B_0^{\pm}$ are given by \eqref{def-B0}.
\end{itemize}
\end{problem}


One can easily see that ${\bf u}={\bf 0}$, $\rho=\rho_0^-$, $p=p_0$ satisfy the following properties:
\begin{itemize}
\item[(i)] (Subsonicity) $|{\bf u}|=0<\sqrt{\gamma p/\rho}=\sqrt{(\gamma p_0)/\rho_0^-};$
\item[(ii)] (Positivity of density) $\rho_0^->0$;
\item[(iii)] As in \eqref{def-B0},
$$\frac{p_0}{(\rho_0^-)^{\gamma}}=S_0^-\quad\mbox{and}\quad \frac{\gamma p_0}{(\gamma-1)\rho_0^-}=B_0^-;$$
\item[(iv)] ${\bf u}\cdot{\bf v}=0$ for any vector ${\bf v}\in\mathbb{R}^2$.
\end{itemize}
From this observation, we fix 
${\bf u}={\bf 0}$, $\rho=\rho_0^-$, $p=p_0$ in $\mathcal{N}\cap\{x_2<g_D(x_1)\}$, 
 and we solve the following free boundary problem to find a solution to Problem \ref{3D-Prob1}.

\begin{figure}[!h]\label{Fig 2D-inf}
\centering
\includegraphics[width=0.99\columnwidth]{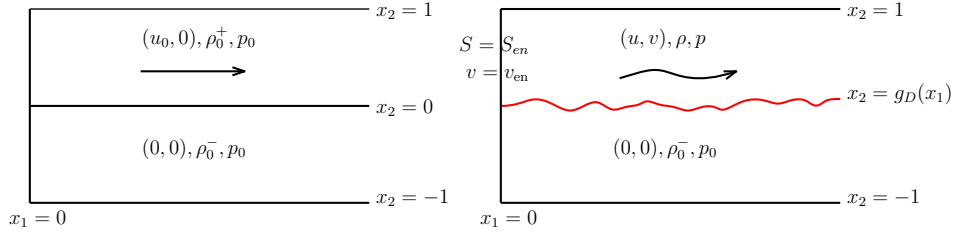}
\caption{Left: Background state, Right: Problem \ref{3D-Prob}}
\end{figure}

\begin{problem}\label{3D-Prob}
Under the same assumptions of Problem \ref{3D-Prob1}, find a function $g_D$ and a $C^1$ solution $U=({\bf u},\rho,p)$ to \eqref{E-S} in $\mathcal{N}_{g_D}^+:=\mathcal{N}\cap\{x_2>g_D(x_1)\}$ such that
\begin{itemize}
\item[(a)] 
\begin{equation}\label{g-00}
g_D(0)=0.
\end{equation}
\item[(b)] Subsonicity: $${|{\bf u}|}<c\quad\mbox{for the sound speed}\quad c=\sqrt{\frac{\gamma p}{\rho}}\quad\mbox{in}\quad\overline{\mathcal{N}^+_{g_D}}.$$
\item[(c)] Positivity of density: $\rho>0$ {in} $\overline{\mathcal{N}^+_{g_D}}.$
\item[(d)] At the entrance $\Gamma_{\rm en}^+$, $U$ satisfies the boundary conditions
\begin{equation}\label{3D-BC-ent}
\begin{split}
\frac{p}{\rho^{\gamma}}=S_{\rm en},\quad {\bf u}\cdot{\bf e}_{2}=v_{\rm en}\quad\mbox{on}\quad \Gamma_{\rm en}^+.
\end{split}
\end{equation}

\item[(e)] On $\Gamma_{g_D}:x_2=g_D(x_1)$, 
$U$ satisfies the boundary conditions
\begin{equation}\label{3D-BC-Cont}
 p=p_0,\quad{\bf u}\cdot {\bf n}_{g_D}=0\quad\mbox{on}\quad \Gamma_{g_D},
\end{equation}
where ${\bf n}_{g_D}$ denotes a unit normal vector field on $\Gamma_{g_D}$.

\item[(f)] On the upper wall $\Gamma_{\rm w}^+$,
$U$ satisfies the slip boundary condition:
\begin{equation}\label{2D-slip}
{\bf u}\cdot{\bf e}_2=0\quad\mbox{on}\quad\Gamma_{\rm w}^+.
\end{equation}
\item[(g)] For $B_0^+$ is given by \eqref{def-B0}, it holds that
$$B({\rm x})=B_0^+ \quad\mbox{in} \quad \overline{\mathcal{N}_{g_D}^+}.$$

\end{itemize}
\end{problem}


Now we state the main theorem of this paper.

\begin{theorem} For given functions 
$S_{\rm en}=S_{\rm en}(x_2)$, $v_{\rm en}=v_{\rm en}(x_2)$ on $\Gamma_{\rm en}$, suppose that they satisfy \eqref{en-epsilon}.
\begin{itemize}
\item[(a)] (Existence)\label{3D-inf-Thm}
 For any fixed $\alpha\in(0,1)$, there exists a small constant $\sigma_1>0$ depending only on $(u_0,\rho_0^+,p_0,S_0^+)$ and $\alpha$ so that if
\begin{equation*}
\sigma=\|S_{\rm en}-S_0\|_{2,\alpha,\Gamma_{\rm en}^+}+\|v_{\rm en}\|_{1,\alpha,\Gamma_{\rm en}^+}\le \sigma_1,
\end{equation*}
 then there exists a solution $U=({\bf u},\rho,p)$ of Problem \ref{3D-Prob} with a contact discontinuity $x_2=g_D(x_1)$ satisfying 
\begin{equation}\label{Th3-est}
\|g_D\|_{2,\alpha,\mathbb{R}^+}+
\|({\bf u},\rho,p)-({\bf u}_0,\rho_0^+,p_0)\|_{1,\alpha,\mathcal{N}^+_{g_D}}
\le C\sigma\quad\text{for ${\bf u}_0:=(u_0,0)$},
\end{equation}
where the constant $C>0$ depends only on $(u_0,\rho_0^+,p_0,S_0^+)$ and $\alpha$.


\item[(b)] (Asymptotic state) \label{3D-Far-Thm}
There exists a constant $\sigma_2\in(0,\sigma_1]$ depending only on $(u_0,\rho_0^+,p_0,S_0^+)$  and $\alpha$ so that if
$${\sigma}\le\sigma_2,$$
then the solution $U=({\bf u},\rho,p)$ in (a) satisfies
\begin{equation*}
\begin{split}
&\lim_{L\to \infty}\|{\bf u}\cdot {\bf e}_2(x_1,\cdot)\|_{C^1(\mathcal{N}\cap\{x_1>L\})}=0,\quad
\lim_{L\to \infty}\|p(x_1,\cdot)-p_0\|_{C^1(\mathcal{N}\cap\{x_1>L\})}=0.
\end{split}
\end{equation*}
\end{itemize}
\end{theorem}

\section{Reformulation of Problem \ref{3D-Prob} via Helmholtz decomposition}\label{3D-Hel}
For ${\bf r}=(r_1,r_2)\in\mathbb{R}^2$, set ${\bf r}^{\bot}=(r_2,-r_1)$.
We express the velocity vector field ${\bf u}$ as 
\begin{equation*}\label{2D-HH}
{\bf u}=\nabla\varphi+\nabla^{\bot}\psi,
\end{equation*}
to  rewrite
the system \eqref{E-S} 
as the following nonlinear system for $(S,\varphi,\psi)$: 
\begin{equation}\label{H-E-S}
\left\{\begin{split}
&\mbox{div}\, {\mathbb F}(S,\nabla\varphi,\nabla\psi)=0,\\
&\Delta\psi=G(S,\partial_{x_2} S,\nabla\varphi,\nabla\psi),\\
&H\left(S,\nabla\varphi,\nabla\psi\right)(\nabla\varphi+\nabla^{\bot}\psi)\cdot\nabla S=0,
\end{split}\right.
\end{equation}
for ${\mathbb F}$, $G$, and $H$ defined by 
\begin{equation}\label{2D_FGH}
\left.\begin{split}
&H(\xi,{\bf q},{\bf r}):=\left[\frac{(\gamma-1)\left(B_0^+-\frac{1}{2}|{\bf q}+{\bf r}^{\bot}|^2\right)}{\gamma\xi}\right]^{1/(\gamma-1)},\\
&{\mathbb F}(\xi, {\bf q},{\bf r}):=H(\xi,{\bf q},{\bf r})({\bf q}+{\bf r}^{\bot}),\quad
G(\xi,\eta,{\bf q},{\bf r}):=-\frac{\eta H^{\gamma-1}(\xi,{\bf q},{\bf r})}{(\gamma-1)(q_1+r_2)},
\end{split}
\right.
\end{equation}
for $\xi,\eta\in\mathbb{R}$, ${\bf q}=(q_1,q_2),{\bf r}=(r_1,r_2)\in\mathbb{R}^2$.

Next, we derive boundary conditions for $(g_D,S,\varphi,\psi)$ to satisfy the physical boundary conditions \eqref{3D-BC-ent}-\eqref{3D-BC-Cont}. In particular, the boundary conditions are derived so that we obtain a compactness of approximated solutions of the free boundary problem given in terms of $(g_D,S,\varphi,\psi)$ right below.

The first step is to rewrite the boundary conditions \eqref{3D-BC-ent}-\eqref{3D-BC-Cont} in terms of $(g_D,S,\varphi,\psi)$ directly.
The boundary conditions \eqref{3D-BC-ent} on $\Gamma_{\rm en}^+$ become
\begin{eqnarray}
\nonumber S=S_{\rm en}\quad&\mbox{on}&\Gamma_{\rm en}^+,\\
\label{ent-BC}
(\nabla\varphi+\nabla^{\bot}\psi) \cdot{\bf e}_2=v_{\rm en}\quad&\mbox{on}&\Gamma_{\rm en}^+.
\end{eqnarray}
We prescribe boundary conditions for $(\varphi,\psi)$ on $\Gamma_{\rm en}^+$ as 
\begin{equation}\label{def-phi-en}
\varphi(x_2)=\int_0^{x_2}v_{\rm en}(t)dt(=:\varphi_{\rm en})\quad\mbox{and}\quad \partial_{x_1}\psi=0
\end{equation}
so that the boundary condition \eqref{ent-BC} holds.
The Rankine-Hugoniot conditions \eqref{3D-BC-Cont} become
\begin{eqnarray}
\label{con-p-BC}
S H^{\gamma}\left(S,\nabla\varphi,\nabla\psi\right)=p_0\quad&\mbox{on}&\Gamma_{g_D},\\
\label{con-n-BC}
(\nabla\varphi+\nabla^{\bot}\psi) \cdot{\bf n}_{g_D}=0\quad&\mbox{on}&\Gamma_{g_D},\quad\mbox{where}\quad{\bf n}_{g_D}=\frac{-g_D'(x_1){\bf e}_1+{\bf e}_2}{\sqrt{1+|g_D'(x_1)|^2}}.
\end{eqnarray}
From  \eqref{g-00} and \eqref{con-n-BC}, we get the following equations for $g_D$:
\begin{equation}\label{Free-BC}
\left\{\begin{split}
&g_D'(x_1)=\frac{(\nabla\varphi+\nabla^{\bot}\psi)\cdot{\bf e}_2}{(\nabla\varphi+\nabla^{\bot}\psi)\cdot{\bf e}_1}(x_1,g_D(x_1))\quad\mbox{for}\quad x_1>0,\\
&g_D(0)=0.
\end{split}\right.
\end{equation}
We use \eqref{Free-BC} to find the location of the contact discontinuity $x_2=g_D(x_1)$.
Since ${\bf u}\cdot{\bf n}_{g_D}=0$ is imposed on $\Gamma_{g_D}$, we have 
\begin{equation}\label{ab-u}
|{\bf u}|^2=|{\bf u}\cdot{\bm\tau}_{g_D}|^2
=\left|(\nabla\varphi+\nabla^{\bot}\psi) \cdot{\bm\tau}_{g_D}\right|^2\quad\mbox{on}\quad\Gamma_{g_D},
\end{equation}
where ${\bm\tau}_{g_D}$ is the unit tangential vector field of $\Gamma_{g_D}$.
Due to the condition (g) stated in Problem \ref{3D-Prob}, the definition of the Bernoulli invariant \eqref{Ber} gives
\begin{equation}\label{3D-u22}
\begin{split}
|{\bf u}|^2
&=2\left(B_0^+-\frac{\gamma p^{1-1/\gamma}S^{1/\gamma}}{\gamma-1}\right)\quad\mbox{on}\quad\Gamma_{g_D}.
\end{split}
\end{equation}
We prescribe boundary conditions for $(\varphi,\psi)$ on $\Gamma_{g_D}$ as 
\begin{equation}\label{varphi-psi-BC}
\nabla\varphi\cdot{\bm\tau}_{g_D}=\nabla\varphi_0\cdot{\bm\tau}_{g_D}\quad\mbox{and}\quad\nabla\psi\cdot{\bf n}_{g_D}=\mathcal{A}(g_D,S)
\end{equation}
for
\begin{eqnarray}
\label{2D-varphi0}&&\varphi_0(x_1,x_2):=u_0x_1,\\
\label{A-def}&&\mathcal{A}(g_D,S):=\sqrt{2\left(B_0^+-\frac{\gamma p_0^{1-1/\gamma}S^{1/\gamma}}{\gamma-1}\right)}-\nabla\varphi_0\cdot{\bm\tau}_{g_D}.
\end{eqnarray}
If \eqref{ab-u}-\eqref{varphi-psi-BC} hold, then the physical boundary condition \eqref{con-p-BC} holds.

We collect the boundary conditions for $(g_D,S,\varphi,\psi)$ with \eqref{Free-BC} as follows:
\begin{equation}\label{2D_BC_2}
\left\{
\begin{split}
S=S_{\rm en},\quad \varphi=\varphi_{\rm en},\quad\partial_{x_1}\psi=0\quad&\mbox{on}\quad\Gamma_{\rm en}^+,\\
\partial_{x_2}\varphi=0,\quad\psi=0\quad&\mbox{on}\quad\Gamma_{\rm w}^+,\\
\nabla\varphi\cdot{\bm \tau}_{g_D}=\nabla\varphi_0\cdot{\bm\tau}_{g_D},\quad\nabla\psi\cdot{\bf n}_{g_D}=\mathcal{A}(g_D,S)\quad&\mbox{on}\quad\Gamma_{g_D}.\\
\end{split}\right.
\end{equation}

\begin{theorem}\label{3D-HD}  
For given functions 
$S_{\rm en}=S_{\rm en}(x_2)$, $v_{\rm en}=v_{\rm en}(x_2)$ on $\Gamma_{\rm en}$, suppose that they satisfy \eqref{en-epsilon}.
For any fixed $\alpha\in(0,1)$, there exists a small constant $\sigma_3>0$ depending only on $(u_0,\rho_0^+,p_0,S_0^+)$ and $\alpha$ so that if
\begin{equation}\label{enu-sigma}
\sigma=\|S_{\rm en}-S_0\|_{2,\alpha,\Gamma_{\rm en}^+}+\|v_{\rm en}\|_{1,\alpha,\Gamma_{\rm en}^+}\le\sigma_3,
\end{equation}
then the free boundary problem \eqref{H-E-S} with boundary conditions \eqref{Free-BC} and \eqref{2D_BC_2} has a solution $(g_D,S,\varphi,\psi)$
 satisfying
\begin{equation}\label{Th4-est}
\|g_D\|_{2,\alpha, \mathbb{R}^+}+\|\varphi-\varphi_0\|_{2,\alpha,\mathcal{N}^+_{g_D}}+\|\psi\|_{2,\alpha,\mathcal{N}^+_{g_D}}+\|S-S_0^+\|_{2,\alpha,\mathcal{N}^+_{g_D}}\le C\sigma,
\end{equation}
where the constant $C>0$ depends only on $(u_0,\rho_0^+,p_0,S_0^+)$ and $\alpha$.
\end{theorem}

Hereafter, a constant $C$ is said to be chosen depending only on the data if $C$ is chosen depending only on $(u_0,\rho_0^+,p_0,S_0^+)$.

In the following, we show that Theorem \ref{3D-inf-Thm} (a) directly follows from Theorem \ref{3D-HD}.
\begin{proof}[{\bf Proof of Theorem \ref{3D-inf-Thm} (a)}] Assume that Theorem \ref{3D-HD} holds true.
For $\sigma_3$ from Theorem \ref{3D-HD}, suppose that the functions $(v_{\rm en}, S_{\rm en})$ satisfy \eqref{enu-sigma}. 
By Theorem \ref{3D-HD}, the free boundary problem \eqref{H-E-S} with \eqref{Free-BC} and \eqref{2D_BC_2} has a solution $(g_D,S, \varphi,\psi)$ that satisfies the estimate \eqref{Th4-est}. For such a solution, we define $({\bf u}, \rho,  p)$ by 
\begin{equation*}
{\bf u}:=\nabla\varphi+\nabla^{\bot}\psi,\quad
\rho:=H(S,\nabla\varphi,\nabla\psi),\quad
p:=S H^{\gamma}(S,\nabla\varphi,\nabla\psi)\quad\text{in $\overline{\mathcal{N}_{g_D}^+}$},
\end{equation*}
for $H$ given by \eqref{2D_FGH}. It follows from the estimate \eqref{Th4-est} given in Theorem \ref{3D-HD} that $(g_D, {\bf u},\rho,p)$ satisfy the estimate \eqref{Th3-est}. Then, one can choose a small constant $\sigma_1\in(0, \sigma_3]$ depending only on the data and $\alpha$ such that if $\sigma\le \sigma_1$, then $(g_D, {\bf u},\rho,p)$ satisfy $\rho\ge \frac 12 \rho_0^+>0$ and $c-|{\bf u}|\ge \frac 12 (c_0^+-u_0^+)>0$ in $\overline{\mathcal{N}_{g_D}^+}$, thus solve Problem \ref{3D-Prob}. Here, $c_0^+$ is given by $c_0^+=\sqrt{\frac{\gamma p_0}{\rho_0^+}}$.
\end{proof}

The rest of the paper is devoted to prove Theorem \ref{3D-HD}  and Theorem \ref{3D-Far-Thm} (b). 
To prove Theorem \ref{3D-HD} by a limiting argument, we introduce a free boundary problem in a cut-off domain of the finite length $L$, and solve it by the method of iteration in the next section. More importantly, uniform estimates of
the solutions to the free boundary problems  in cut-off domains are established independently of the length $L$ in the next section. Then, we take a  sequence of the solutions to the free boundary problems, and pass to the limit $L\to \infty$ in \S \ref{subsec-ex}. Then, the limit yields a solution to the free boundary problem \eqref{H-E-S} with boundary conditions \eqref{Free-BC} and \eqref{2D_BC_2}.

\section{Free boundary problems in cut-off domains}\label{sec-cut}

\subsection{Iteration scheme}
Fix a constant $L>0$ and define $\mathcal{N}_L$ by
\begin{equation}\label{def-NL}
\mathcal{N}_L:=\mathcal{N}\cap\{0<x_1<L\}
\end{equation}
for $\mathcal{N}$ given by \eqref{definition-N}.
And, we define  $\Gamma_{\rm w}^{+,L}$ and $\Gamma_{\rm ex}^L$  by
$$\Gamma_{\rm w}^{+,L}:=\partial\mathcal{N}_L\cap\{x_2=1\},\quad \Gamma_{\rm ex}^L:=\partial\mathcal{N}_L\cap\{x_1=L\}.$$
For a function $f:[0,L]\rightarrow (-1,1)$, we set
\begin{equation*}
\begin{split}
&\mathcal{N}_{L,f}^+:=\mathcal{N}_L\cap\{x_2>f(x_1)\},\\
&\Gamma^{L,f}_{\rm ex}:=\partial\mathcal{N}_{L,f}^+\cap\{x_1=L\},\quad
\Gamma_{\rm cd}^{L,f}:=\partial\mathcal{N}_{L,f}^+\cap\{x_2=f(x_1)\}.
\end{split}
\end{equation*}

\begin{problem} \label{Prob-Cut}
Find a solution $(f,S,\varphi,\psi)$ of the following free boundary problem: 
\begin{equation}\label{3D-Cut-Eq}
\eqref{H-E-S}\quad\mbox{in}\quad\mathcal{N}_{L,f}^+
\end{equation}
 with boundary conditions
\begin{equation}\label{2D_BC_Fin}\left\{
\begin{split}
S=S_{\rm en},\quad \varphi=\varphi_{\rm en},\quad\partial_{x_1}\psi=0\quad&\mbox{on}\quad\Gamma_{\rm en}^+,\\
\partial_{x_2}\varphi=0,\quad\psi=0\quad&\mbox{on}\quad\Gamma_{\rm w}^{+,L},\\
\partial_{x_2}\varphi=0,\quad\partial_{x_1}\psi=0\quad&\mbox{on}\quad\Gamma^{L,f}_{\rm ex},\\
\nabla\varphi\cdot{\bm \tau}_f=\nabla\varphi_0\cdot{\bm\tau}_f,\quad\nabla\psi\cdot{\bf n}_{f}=\mathcal{A}(f,S)\quad&\mbox{on}\quad\Gamma_{\rm cd}^{L,f},\\
\end{split}\right.
\end{equation}
\begin{equation}\label{2D_BC_Fin2}
\left\{\begin{split}
&f'(x_1)=\frac{(\nabla\varphi+\nabla^{\bot}\psi)\cdot{\bf e}_2}{(\nabla\varphi+\nabla^{\bot}\psi)\cdot{\bf e}_1}(x_1,f(x_1))\quad\mbox{for}\quad x_1>0,\\
&f(0)=0,
\end{split}\right.
\end{equation}
where ${\bm\tau}_f$ and ${\bf n}_f$ are the unit tangential vector field and unit normal vector field of $\Gamma_{\rm cd}^{L,f}$, respectively.
\end{problem}

The goal of Section \ref{sec-cut} is to prove the following proposition.

\begin{proposition}\label{3D-Prop}
For given functions 
$S_{\rm en}=S_{\rm en}(x_2)$, $v_{\rm en}=v_{\rm en}(x_2)$ on $\Gamma_{\rm en}$, suppose that they satisfy \eqref{en-epsilon}.
For a fixed $\alpha\in(0,1)$, there exists a small constant $\sigma_4>0$ depending only on the data and $\alpha$ so that if
$$\sigma=\|S_{\rm en}-S_0\|_{2,\alpha,\Gamma_{\rm en}^+}+\|v_{\rm en}\|_{1,\alpha,\Gamma_{\rm en}^+}\le\sigma_4,$$
then Problem \ref{Prob-Cut}  has  a unique solution $(f,S,\varphi,\psi)$ satisfying
\begin{equation}\label{3D-Prop-est}
\|f\|_{2,\alpha,(0,L)}+\|\varphi-\varphi_0\|_{2,\alpha,\mathcal{N}_{L,f}^+}+\|\psi\|_{2,\alpha,\mathcal{N}_{L,f}^+}+\|S-S_0^+\|_{2,\alpha,\mathcal{N}_{L,f}^+}\le C\sigma,
\end{equation}
where the constant $C>0$ depends only on the data and $\alpha$ but independent of $L$.
\end{proposition}

To find the entropy $S$ of Problem \ref{Prob-Cut}, we have to solve a transport equation in $\mathcal{N}_{L,f}^+$. 
According to \cite[Lemma 3.3]{B-D-X}, the condition \eqref{2D_BC_Fin2} must hold to be able to solve the transport equation for $S$.
Thus we first solve a free boundary problem for a fixed approximated entropy. 

For a fixed constant $\alpha\in(0,1)$ and a constant $M_1>0$ to be determined later, we define an iteration set
\begin{equation}\label{S-set}
\mathcal{S}(M_1):=\left\{S\in C^{2,\alpha}(\overline{\mathcal{N}_{L,-1/2}^+}):
\begin{split}
&\|S-S_0^+\|_{2,\alpha,\mathcal{N}_{L,-1/2}^+}\le M_1\sigma,\\
&\partial_{x_1}S\equiv0\mbox{ on }\Gamma_{\rm en}^\epsilon\cup\Gamma_{\rm ex}^L
\end{split}\right\},
\end{equation}
for $\mathcal{N}_{L,-1/2}^+:=\mathcal{N}_L\cap\left\{x_2>-{1}/{2}\right\}$, and $\Gamma_{\rm en}^\epsilon$ defined by \eqref{en-epsilon}.

\begin{problem}\label{FBP-S}
For each $S_{\ast}\in\mathcal{S}(M_1)$, find $(f, \varphi,\psi)$ satisfying \eqref{2D_BC_Fin2} and
\begin{equation}\label{S-free}
\left\{
\begin{split}
\mbox{div }{\mathbb F}(S_{\ast},\nabla\varphi,\nabla\psi)=0,\quad
	\Delta\psi=G(S_{\ast},\partial_{x_2}S_{\ast},\nabla\varphi,\nabla\psi)\quad&\mbox{in}\quad\mathcal{N}_{L,f}^+,\\
\varphi=\varphi_{\rm en},\quad\partial_{x_1}\psi=0\quad&\mbox{on}\quad\Gamma_{\rm en}^+,\\
\partial_{x_2}\varphi=0,\quad\psi=0\quad&\mbox{on}\quad\Gamma_{\rm w}^{+,L},\\
\partial_{x_2}\varphi=0,\quad\partial_{x_1}\psi=0\quad&\mbox{on}\quad\Gamma^{L,f}_{\rm ex},\\
\nabla\varphi\cdot{\bm \tau}_f=\nabla\varphi_0\cdot{\bm\tau}_f,\quad\nabla\psi\cdot{\bf n}_{f}=\mathcal{A}(f,S_{\ast})\quad&\mbox{on}\quad\Gamma_{\rm cd}^{L,f},
\end{split}\right.
\end{equation}
where $\mathbb{F}$, $G$, and $\mathcal{A}$ are given by \eqref{2D_FGH} and \eqref{A-def}.
\end{problem}

\begin{lemma}\label{S-free-2}
Under the same assumptions on $(S_{\rm en}, v_{\rm en})$ as in {\emph{Proposition \ref{3D-Prop}}},
there exists a small constant $\sigma_5>0$ depending only on the data and $(\alpha, M_1)$ so that if 
$${\sigma}=\|S_{\rm en}-S_0\|_{2,\alpha,\Gamma_{\rm en}^+}+\|v_{\rm en}\|_{1,\alpha,\Gamma_{\rm en}^+}\le\sigma_5,$$
then, for each $S_{\ast}\in\mathcal{S}(M_1)$, Problem \ref{FBP-S}  has a unique solution $(f, \varphi,\psi)$ satisfying
\begin{equation}\label{S-free-est}
 \|f\|_{2,\alpha,(0,L)}+
\|(\varphi,\psi)-(\varphi_0,0)\|_{2,\alpha,\mathcal{N}_{L,f}^+}\le C\left(M_1+1\right)\sigma,
\end{equation}
where the constant $C>0$ depends only on the data and $\alpha$ but independent of $L$.
\end{lemma}

Once Lemma \ref{S-free-2} is proved, we prove Proposition \ref{3D-Prop} by the following approach:
For a fixed $S_\ast\in\mathcal{S}(M_1)$, let $(f,\varphi,\psi)$ be the unique solution of Problem \ref{FBP-S} with satisfying the estimate \eqref{S-free-est}. For such a solution, we solve  the following initial value problem for $S$:
\begin{equation}\label{Ite-3D2}
\left\{\begin{split}
	&H(S_{\ast},\nabla\varphi,\nabla\psi)(\nabla\varphi+\nabla^{\bot}\psi)\cdot \nabla S=0\quad\mbox{in}\quad\mathcal{N}_{L,f}^+,\\
	&S=S_{\rm en}\quad\mbox{on}\quad\Gamma_{\rm en}^+.
\end{split}\right.
\end{equation}
Take a suitable extension $S^f\in C^{2,\alpha/2}(\overline{\mathcal{N}_{L,-1/2}^+})$ of $S$. 
For such $S^f$, we define an  iteration mapping $\mathcal{J}:\mathcal{S}(M_1)\rightarrow C^{2,\alpha/2}(\overline{\mathcal{N}_{L,-1/2}^+})$ by 
$$\mathcal{J}(S_{\ast})=S^f.$$
Then we choose $M_1$ and $\sigma$ so that the mapping $\mathcal{J}$ maps $\mathcal{S}(M_1)$ into itself.
We show that there exists a fixed point $S^f$ of $\mathcal{J}$ by the Schauder fixed point theorem.
Finally, we prove that $(f,S,\varphi,\psi)$ is the unique solution of Problem \ref{Prob-Cut}, and satisfies the estimate \eqref{3D-Prop-est} stated in Proposition  \ref{3D-Prop}. Details are given in the rest of this section.

\subsection{Proof of Lemma \ref{S-free-2}}
For a constant $M_2>0$ to be determined later with $M_2\sigma\le\frac{1}{4}$, we define an iteration set
\begin{equation}\label{f-set}
\mathcal{F}(M_2):=\left\{f\in C^{2,\alpha}([0,L]):
	\begin{split}
	&\|f\|_{2,\alpha,(0,L)}\le M_2\sigma,\\
	&f'(0)=f'(L)=0, \mbox{ }f(0)=0
	\end{split}\right\}.
\end{equation}
We fix $f_{\ast}\in\mathcal{F}(M_2)$, and solve the following boundary value problem in $\mathcal{N}_{L,f_{\ast}}^+$:
\begin{equation}\label{Ite-3D}
\left\{
\begin{split}
\mbox{div}{\mathbb F}(S_{\ast},\nabla\varphi,\nabla\psi)=0,\quad
	\Delta\psi=G(S_{\ast},\partial_{x_2}S_{\ast},\nabla\varphi,\nabla\psi)\quad&\mbox{in}\quad\mathcal{N}_{L,f_{\ast}}^+,\\
\varphi=\varphi_{\rm en},\quad\partial_{x_1}\psi=0\quad&\mbox{on}\quad\Gamma_{\rm en}^+,\\
\partial_{x_2}\varphi=0,\quad\psi=0\quad&\mbox{on}\quad\Gamma_{\rm w}^{+,L},\\
\partial_{x_2}\varphi=0,\quad\partial_{x_1}\psi=0\quad&\mbox{on}\quad\Gamma^{L,f_{\ast}}_{\rm ex},\\
\nabla\varphi\cdot{\bm \tau}_{f_{\ast}}=\nabla\varphi_0\cdot{\bm\tau}_{f_{\ast}},\quad\nabla\psi\cdot{\bf n}_{f_{\ast}}=\mathcal{A}(f_{\ast},S_{\ast})\quad&\mbox{on}\quad\Gamma_{\rm cd}^{L,f_{\ast}},
\end{split}\right.
\end{equation}
for $\mathbb{F}$, $G$, and $\mathcal{A}$ given by \eqref{2D_FGH} and \eqref{A-def}.

\begin{lemma}\label{2D-fix}
Under the same assumptions on $(S_{\rm en}, v_{\rm en})$ as in {\emph{Proposition \ref{3D-Prop}}}, there exists a small constant $\sigma_6>0$ depending only on the data and $(\alpha, M_1, M_2)$ so that if 
$${\sigma}=\|S_{\rm en}-S_0\|_{2,\alpha,\Gamma_{\rm en}^+}+\|v_{\rm en}\|_{1,\alpha,\Gamma_{\rm en}^+}\le\sigma_6,$$
then the boundary value problem \eqref{Ite-3D} has a unique solution $(\varphi,\psi)$ satisfying
\begin{equation}\label{2D-Fix-est}
\begin{split}
\|(\varphi,\psi)-(\varphi_0,0)\|_{2,\alpha,\mathcal{N}_{L,f_\ast}^+}&\le C\left(M_1+1\right)\sigma,
\end{split}
\end{equation}
where the constant $C>0$ depends only on the data and $\alpha$ but independent of $L$.
\end{lemma}

\begin{proof}
{\bf 1.} 
For two constants $M_3,M_4>0$  to be determined later, we define an iteration set 
$$\mathcal{K}^{f_\ast}(M_3,M_4):=\mathcal{K}_1^{f_\ast}(M_3)\times\mathcal{K}_2^{f_{\ast}}(M_4)$$
for 
\begin{equation}\label{phi-set}
\left.\begin{split}
&\mathcal{K}_1^{f_\ast}(M_3):=\left\{\phi\in C^{2,\alpha}(\overline{\mathcal{N}_{L,f_\ast}^+}): 
	\begin{split}
	&\|\phi\|_{2,\alpha,\mathcal{N}_{L,f_\ast}^+}\le M_3\sigma,\mbox{ }\phi=\varphi_{\rm en}\mbox{ on }\Gamma_{\rm en}^+,\\
	&\partial_{x_1x_1}\phi=0\mbox{ on }\Gamma_{\rm en}^{\epsilon}\cup\Gamma^{L,f_{\ast}}_{\rm ex},\\
	&\phi=0\mbox{ on }\Gamma^{L,f_{\ast}}_{\rm ex}
	\end{split}\right\},\\
&\mathcal{K}_2^{f_\ast}(M_4):=\left\{\psi\in C^{2,\alpha}(\overline{\mathcal{N}_{L,f_\ast}^+}):
	\begin{split}
	&\|\psi\|_{2,\alpha,\mathcal{N}_{L,f_\ast}^+}\le M_4\sigma,\\
	&\partial_{x_1}\psi=0\mbox{ on }\Gamma_{\rm en}^+\cup\Gamma^{L,f_{\ast}}_{\rm ex}
	\end{split}\right\},
\end{split}\right.
\end{equation}
where $\Gamma_{\rm en}^{\epsilon}$ is given by \eqref{en-epsilon}.

{\bf 2.} 
For a fixed $(\tilde{\phi},\tilde{\psi})\in\mathcal{K}^{f_\ast}(M_3,M_4)$, let
\begin{equation}\label{GB-def}
\mathfrak{G}:=G(S_{\ast},\partial_{x_2}S_{\ast},\nabla\tilde{\phi}+\nabla\varphi_0,\nabla\tilde{\psi}),\quad
\mathfrak{B}:=\mathcal{A}(f_{\ast},S_{\ast}),
\end{equation}
where $G$ and $\mathcal{A}$ are given by \eqref{2D_FGH} and \eqref{A-def}, respectively. 
By the standard elliptic theory, the linear boundary value problem 
\begin{equation}\label{lin-psi}
\left\{
\begin{split}
\Delta\psi=\mathfrak{G}\quad&\mbox{in}\quad\mathcal{N}_{L,f_{\ast}}^+,\\
\psi=0\quad&\mbox{on}\quad\Gamma_{\rm w}^{+,L},\\
\partial_{x_1}\psi=0\quad&\mbox{on}\quad\Gamma_{\rm en}^+\cup\Gamma^{L,f_{\ast}}_{\rm ex},\\
\nabla\psi\cdot{\bf n}_{f_\ast}=\mathfrak{B}\quad&\mbox{on}\quad\Gamma_{\rm cd}^{L,f_{\ast}},
\end{split}\right.
\end{equation}
has a unique solution $\psi\in C^{1,\alpha}(\overline{\mathcal{N}_{L,f_{\ast}}^+})\cap C^{2,\alpha}({\mathcal{N}_{L,f_\ast}^+})$.

{\bf Claim:} $\psi$ satisfies the estimate
\begin{equation}\label{psi-est-2D}
\|\psi\|_{2,\alpha,\mathcal{N}_{L,f_\ast}^+}\le C\left(\|\mathfrak{G}\|_{\alpha,\mathcal{N}_{L,f_\ast}^+}+\|\mathfrak{B}\|_{1,\alpha,\Gamma_{\rm cd}^{L,f_{\ast}}}\right),
\end{equation}
where the constant $C>0$ depends only on the data and $\alpha$ but independent of $L$.

Hereafter, any estimate constant $C$ is regarded to be chosen depending only on the data and $\alpha$ but independent of $L$ unless specified otherwise.

\begin{proof}[Proof of Claim]
Set
$$\mathfrak{a}:=\|\mathfrak{G}\|_{0,\mathcal{N}_{L,f_\ast}^+}+\|\mathfrak{B}\|_{0,\Gamma_{\rm cd}^{L,f_{\ast}}},$$
and define a function $\mathfrak{N}$ by 
$$\mathfrak{N}(x_1,x_2):=2e\mathfrak{a}\left(-e^{x_2}-{x_2}+1+e\right) \quad\mbox{for}\quad(x_1,x_2)\in\overline{\mathcal{N}_{L,f_\ast}^+}.$$
By a direct computation, one can directly check that $\mathfrak{N}\pm\psi$ satisfies 
\begin{equation*}
\left\{\begin{split}
\Delta(\mathfrak{N}\pm\psi)=-\mathfrak{a}e^{1+{x_2}}\pm\mathfrak{G}\le 0\quad&\mbox{in}\quad\mathcal{N}_{L,f_\ast}^+,\\
\mathfrak{N}\pm\psi=0\quad&\mbox{on}\quad\Gamma_{\rm w}^{+,L},\\
\partial_{x_1}(\mathfrak{N}\pm \psi)=0\quad&\mbox{on}\quad\Gamma_{\rm en}^+\cup\Gamma^{L,f_{\ast}}_{\rm ex},\\
\nabla(\mathfrak{N}\pm\psi)\cdot{\bf n}_{f_\ast}=-\frac{2e\mathfrak{a}(-e^{x_2}-1)}{\sqrt{1+|f'_{\ast}(x_1)|^2}}\pm\mathfrak{B}\ge 0\quad&\mbox{on}\quad\Gamma_{\rm cd}^{L,f_{\ast}}.
\end{split}\right.
\end{equation*}
By the comparison principle and Hopf's lemma, we have
$-\mathfrak{N}\le\psi\le\mathfrak{N}.$
Thus we have
$$\|\psi\|_{0,\mathcal{N}_{L,f_\ast}^+}\le C\left(\|\mathfrak{G}\|_{0,\mathcal{N}_{L,f_\ast}^+}+\|\mathfrak{B}\|_{0,\Gamma_{\rm cd}^{L,f_{\ast}}}\right).$$

To obtain $C^{2,\alpha}$ estimate up to the boundary, 
we use the method of reflection. Define an extension of $f_{\ast}\in\mathcal{F}(M_2)$ into $-1\le x_1\le L+1$ by 
\begin{equation*}
f_{\ast}^e(x_1):=\left\{\begin{split}
f_{\ast}(-x_1)\quad\mbox{for}\quad&-1\le x_1<0,\\
f_{\ast}(x_1)\quad\mbox{for}\quad&0\le x_1\le L,\\
f_{\ast}(2L-x_1)\quad\mbox{for}\quad &L<x_1\le L+1.
\end{split}\right.
\end{equation*}
Then $f_{\ast}^e\in C^{2,\alpha}([-1,L+1])$ since $f_{\ast}'(0)=f_{\ast}'(L)=0$.
We define an extended domain
\begin{equation*}\label{2D-N-ext}
\mathcal{N}_{\rm ext}:=\left\{(x_1,x_2)\in\mathbb{R}^2: -1<x_1<L+1, f_{\ast}^e(x_1)<x_2<1\right\},
\end{equation*}
and even extensions of $\psi$, $\mathfrak{G}$, and $\mathfrak{B}$ into $\mathcal{N}_{\rm ext}$ as follows:
\begin{equation*}
\begin{split}
\left(\psi_{\rm ext},\mathfrak{G}_{\rm ext}\right)(x_1,x_2)
&:=\left\{\begin{split}
	\left(\psi,\mathfrak{G}\right)(-x_1,x_2)\quad&\mbox{for }-1\le x_1\le0,\\
	\left(\psi,\mathfrak{G}\right)(x_1,x_2)\quad&\mbox{for }0\le x_1\le L,\\
	\left(\psi,\mathfrak{G}\right)(2L-x_1,x_2)\quad&\mbox{for }L\le x_1\le L+1,\\
	\end{split}\right.\\
\mathfrak{B}_{\rm ext}(x_1,x_2)
&:=\left\{\begin{split}
	\mathfrak{B}(-x_1,f_{\ast}^e(x_1))\quad&\mbox{for }-1\le x_1\le0,\\
	\mathfrak{B}(x_1,f_{\ast}^e(x_1))\quad&\mbox{for }0\le x_1\le L,\\
	\mathfrak{B}(2L-x_1,f_{\ast}^e(x_1))\quad&\mbox{for }L\le x_1\le L+1.
	\end{split}\right.
\end{split}
\end{equation*}
Then $\mathfrak{G}_{\rm ext}\in C^{\alpha}(\overline{\mathcal{N}_{\rm ext}})$ satisfies
\begin{equation}\label{2D-G-est}
\|\mathfrak{G}_{\rm ext}\|_{\alpha,\mathcal{N}_{\rm ext}}\le C\|\mathfrak{G}\|_{\alpha,\mathcal{N}_{L,f_\ast}^+}.
\end{equation}
By the compatibility conditions of $(S_{\ast},f_{\ast})$ given in \eqref{S-set} and \eqref{f-set}, 
$$\nabla\mathfrak{B}_{\rm ext}\cdot{\bm \tau}_{f_{\ast}}(0,0)=\nabla\mathfrak{B}_{\rm ext}\cdot{\bm \tau}_{f_\ast}(L,f_{\ast}(L))=0,$$
where ${\bm \tau}_{f_\ast}$ is the unit tangential vector field of $\partial\mathcal{N}_{\rm ext}\cap\{x_2=f_{\ast}^e(x_1)\}$.
From this and the definition of $\mathfrak{B}_{\rm ext}$, we have the estimate
\begin{equation}\label{2D-B-est}
\|\mathfrak{B}_{\rm ext}\|_{1,\alpha,\partial\mathcal{N}_{\rm ext}\cap\{x_2=f_{\ast}^e(x_1)\}}\le C\|\mathfrak{B}\|_{1,\alpha,\Gamma_{\rm cd}^{L,f_{\ast}}}.
\end{equation}

Consider a connected subdomain $\mathcal{N}_l$ of ${\mathcal{N}_{\rm ext}}$ such that 
\begin{equation*}\label{2D-NL}
\mathcal{N}_{\rm ext}\cap\left\{-\frac{1}{2}\le x_1\le\frac{1}{2}\right\}\subset{\mathcal{N}_l}\subset\mathcal{N}_{\rm ext}\cap\left\{-{1}\le x_1\le{1}\right\}
\end{equation*}
and $\partial{\mathcal{N}_l}$ is a simple closed smooth curve.
By the standard elliptic theory, the boundary value problem 
\begin{equation}\label{ref-Psi}
\left\{\begin{split}
\Delta\Psi=\mathfrak{G}_{\rm ext}\quad&\mbox{in}\quad{\mathcal{N}_{l}},\\
\Psi=0\quad&\mbox{on}\quad\partial{\mathcal{N}_{l}}\cap\{x_2=1\},\\
\nabla\Psi\cdot{\bf n}_{f_\ast}=\mathfrak{B}_{\rm ext}\quad&\mbox{on}\quad\partial{\mathcal{N}_{l}}\cap\{x_2=f^e_{\ast}(x_1)\},\\
\Psi=\psi_{\rm ext}\quad&\mbox{on}\quad\partial{\mathcal{N}_{l}}\backslash\{x_2=1,f^e_{\ast}(x_1)\},\\
\end{split}\right.
\end{equation}
 has a unique solution $\Psi\in C^{2,\alpha}(\overline{\mathcal{N}_{l}})$ satisfying 
\begin{equation*}
\|\Psi\|_{2,\alpha,{\mathcal{N}_{l}}}\le C\left(\|\mathfrak{G}_{\rm ext}\|_{\alpha,{\mathcal{N}_{l}}}+\|\mathfrak{B}_{\rm ext}\|_{1,\alpha,\partial{\mathcal{N}_{l}}\cap\Gamma_{\rm cd}^{L,f_{\ast}}}\right).
\end{equation*}
It follows from \eqref{2D-G-est}-\eqref{2D-B-est} that
\begin{equation*}
\|\Psi\|_{2,\alpha,{\mathcal{N}_{l}}}\le C\left(\|\mathfrak{G}\|_{\alpha,\mathcal{N}_{L,f_\ast}^+}+\|\mathfrak{B}\|_{1,\alpha,\Gamma_{\rm cd}^{L,f_{\ast}}}\right).
\end{equation*}
By the definitions of $(\mathfrak{G}_{\rm ext}, \mathfrak{B}_{\rm ext}, \psi_{\rm ext})$ and the uniqueness of a solution to \eqref{ref-Psi}, we have 
 $\Psi(x_1,x_2)=\Psi(-x_1,x_2)$ and $\partial_{x_1}\Psi(0,x_2)=0$.
 Then
$\Psi(x_1,x_2)=\psi(x_1,x_2)$ for $(x_1,x_2)\in\mathcal{N}_l\cap\{x_1\ge0\}$
by the uniqueness of a solution to \eqref{lin-psi}.
Thus we get the estimate
\begin{equation}\label{psi-left}
\|\psi\|_{2,\alpha,\mathcal{N}_{l}\cap\mathcal{N}_{L,f_\ast}^+}\le C\left(\|\mathfrak{G}\|_{\alpha,\mathcal{N}_{L,f_\ast}^+}+\|\mathfrak{B}\|_{1,\alpha,\Gamma_{\rm cd}^{L,f_{\ast}}}\right).
\end{equation}
One can also similarly check that 
\begin{equation}\label{psi-right}
\|\psi\|_{2,\alpha,\mathcal{N}_{L,f_\ast}^+\cap\{x_1\ge L-1/2\}}\le C\left(\|\mathfrak{G}\|_{\alpha,\mathcal{N}_{L,f_\ast}^+}+\|\mathfrak{B}\|_{1,\alpha,\Gamma_{\rm cd}^{L,f_{\ast}}}\right).
\end{equation}
By \eqref{psi-left} and \eqref{psi-right}, we have the $C^{2,\alpha}$ estimate
\begin{equation*}
\|\psi\|_{2,\alpha,\mathcal{N}_{L,f_\ast}^+}\le C\left(\|\mathfrak{G}\|_{\alpha,\mathcal{N}_{L,f_\ast}^+}+\|\mathfrak{B}\|_{1,\alpha,\Gamma_{\rm cd}^{L,f_{\ast}}}\right).
\end{equation*}
The proof of Claim is completed.
\end{proof}

{\bf 3.} 
For $\xi\in\mathbb{R}$, ${\bf q}=(q_1,q_2)\in\mathbb{R}^2$, and ${\bf r}\in\mathbb{R}^2$,
define ${\bf A}=(A_1,A_2)$ by 
\begin{equation*}\left.
\begin{split}
&A_j(\xi,{\bf q},{\bf r}):=H(\xi,{\bf q},{\bf r})q_j\quad\mbox{for}\quad j=1,2,
\end{split}\right.
\end{equation*}
where $H$ is given by \eqref{2D_FGH}.
Then the first equation $\mbox{div}\,{\mathbb F}(S_{\ast},\nabla\varphi,\nabla\psi)=0$ in \eqref{Ite-3D} can be rewritten as 
\begin{equation}\label{varphi-eq-1}
\mbox{div}\left({\bf A}(S_{\ast},\nabla\varphi,\nabla\psi)\right)=-\mbox{div}\left(H(S_{\ast},\nabla\varphi,\nabla\psi)(\nabla^{\bot}\psi)\right).
\end{equation}
For $\varphi_0$ given by \eqref{2D-varphi0}, denote ${\bf V}_0:=(S_0^+,\nabla\varphi_0,{\bf 0})$ and set
\begin{equation}\label{2D-aij}
a_{ij}:=\partial_{q_j}A_i({\bf V}_0)\quad\mbox{for}\quad i,j=1,2.
\end{equation}
Then the matrix $[a_{ij}]_{i,j=1}^2$ is strictly positive and diagonal, and there exists a constant $\nu\in(0,{1}/{10}]$ satisfying 
\begin{equation}\label{nu}
\nu \mathbb{I}_2\le[a_{ij}]_{i,j=1}^2\le \frac{1}{\nu}\mathbb{I}_2.
\end{equation}

Set $\phi:=\varphi-\varphi_0$ for $\varphi_0$ given by \eqref{2D-varphi0}.
Then \eqref{varphi-eq-1} can be rewritten as 
$$\mathcal{L}(\phi)=\mbox{div}{\bf F}(S_{\ast}-S_0^+,\nabla\phi,\nabla\psi),$$
where $\mathcal{L}$ and ${\bf F}=(F_1,F_2)$ are defined as follows:
\begin{equation}\label{2D-def-F}
\left.
\begin{split}
&\mathcal{L}(\phi)
:=\sum_{i=1}^2a_{ii}\partial_{ii}\phi,\\
&\begin{split}F_i(Q):=&-\int_0^1 D_{(\xi,\bf r)}A_i({\bf V}_0+tQ)dt\cdot(\xi,{\bf r})\\
				&-\int_0^1 D_{\bf q}A_i({\bf V}_0+tQ)-D_{\bf q}A_i({\bf V}_0)dt\cdot{\bf q}-H({\bf V}_0+Q)({\bf r}^{\bot})_i,
				\end{split}
\end{split}\right.
\end{equation}
with $Q=(\xi,{\bf q},{\bf r})\in\mathbb{R}\times(\mathbb{R}^2)^2$. 
Here, $\partial_{x_i}$ is abbreviated as $\partial_i$.

By the boundary conditions for $\varphi$ given in \eqref{Ite-3D} and the definition of $\varphi_0$ given by \eqref{2D-varphi0},  the boundary conditions for $\phi$ on 
$\partial \mathcal{N}_{L, f_*}^+\setminus \Gamma_{\rm cd}^{L, f_*}$ become
\begin{equation*}
\phi=\varphi_{en}\quad\mbox{on}\quad \Gamma_{\rm en}^+,\quad
\partial_{x_2}\phi=0\quad\mbox{on}\quad \Gamma_{\rm w}^{+,L},\quad
\phi=0\quad\mbox{on}\quad \Gamma_{\rm ex}^{L,f_{\ast}}.
\end{equation*}
On $\Gamma_{\rm cd}^{L,f_{\ast}}$, the boundary condition for $\varphi$  given in \eqref{Ite-3D} implies that $\phi$ should be a constant along $\Gamma_{\rm cd}^{L,f_{\ast}}$. Since we seek a solution $\phi$ to be continuous up to the boundary, and since $\varphi_{\rm en}(0,0)=0$ by the definition \eqref{def-phi-en}, we prescribe the boundary condition for $\phi$ on $\Gamma_{\rm cd}^{L,f_{\ast}}$ as
\begin{equation*}
\phi=0\quad\mbox{on}\quad \Gamma_{\rm cd}^{L,f_{\ast}}.
\end{equation*}

{\bf 4.} With the unique solution $\psi\in C^{2,\alpha}(\overline{\mathcal{N}_{L,f_\ast}^+})$ to \eqref{lin-psi} associated with $(\tilde{\phi},\tilde{\psi})\in\mathcal{K}^{f_\ast}(M_3,M_4)$, set
\begin{equation}\label{FF_def}
\mathfrak{F}:=\mbox{div}{\bf F}(S_{\ast}-S_0^+,\nabla\tilde{\phi},\nabla\psi),
\end{equation}
where ${\bf F}$ is given by \eqref{2D-def-F}.

{\bf Claim:}
The linear boundary value problem
\begin{equation}\label{lin-phi}
\left\{\begin{split}
\mathcal{L}(\phi)=\mathfrak{F}\quad&\mbox{in}\quad \mathcal{N}_{L,f_\ast}^+,\\
\phi=\varphi_{\rm en}\quad&\mbox{on}\quad\Gamma_{\rm en}^+,\\
\phi=0\quad&\mbox{on}\quad \Gamma_{\rm cd}^{L,f_{\ast}}\cup\Gamma^{L,f_{\ast}}_{\rm ex},\\
\partial_{x_2}\phi=0\quad&\mbox{on}\quad \Gamma_{\rm w}^{+,L},
\end{split}\right.
\end{equation}
has a unique solution $\phi\in C^{2,\alpha}(\overline{\mathcal{N}_{L,f_\ast}^+})$. 
Moreover, the solution $\phi$ satisfies
$$\partial_{x_1x_1}\phi\equiv 0\quad\mbox{on}\quad\Gamma_{\rm en}^\epsilon\cup\Gamma^{L,f_{\ast}}_{\rm ex},$$ 
where $\Gamma_{\rm en}^{\epsilon}$ is given by \eqref{en-epsilon},
and the estimate
\begin{equation}\label{phi-est-2D}
\|\phi\|_{2,\alpha,\mathcal{N}_{L,f_\ast}^+}
\le C\left(\|\mathfrak{F}\|_{\alpha,\mathcal{N}_{L,f_\ast}^+}+\|{\varphi_{\rm en}}\|_{2,\alpha,\Gamma_{\rm en}^+}\right),
\end{equation}
where the constant $C>0$ depends only on the data and $\alpha$ but independent of $L$.

\begin{proof}[Proof of Claim]
For $\varphi_{\rm en}$ given by \eqref{def-phi-en}, define a function ${\varphi_{\rm en}^{\ast}}$ by 
\begin{equation}\label{def-var-en}
{\varphi_{\rm en}^{\ast}}(x_1,x_2):=\eta(x_1)\varphi_{\rm en}\left(\frac{x_2-f_{\ast}(x_1)}{1-f_{\ast}(x_1)}\right)\quad\mbox{for}\quad(x_1,x_2)\in\mathcal{N}_{L,f_{\ast}}^+,
\end{equation}
where $\eta$ is a $C^{\infty}$ function satisfying 
\begin{equation}\label{def-eta}
\eta=1\quad\mbox{for } x_1<\frac{L}{3},\quad\eta=0\quad\mbox{for } x_1>\frac{2L}{3},\quad |\eta'(x_1)|\le 2,\quad|\eta''(x_1)|\le2.
\end{equation}
Set $\phi_{\rm hom}:=\phi-{\varphi_{\rm en}^{\ast}}$.
Then \eqref{lin-phi} can be rewritten as 
\begin{equation}\label{lin-phi-hom}
\left\{\begin{split}
\mathcal{L}(\phi_{\rm hom})={\mathfrak{F}^{\ast}}&\quad\mbox{in}\quad\mathcal{N}_{L,f_{\ast}}^+,\\
\phi_{\rm hom}=0&\quad\mbox{on}\quad \partial\mathcal{N}_{L,f_{\ast}}^+\backslash\Gamma_{\rm w}^{+,L},\\
\partial_{x_2}\phi_{\rm hom}=0&\quad\mbox{on}\quad \Gamma_{\rm w}^{+,L},\\
\end{split}\right.
\end{equation}
where the function ${\mathfrak{F}^{\ast}}$ is defined by 
\begin{equation}\label{Fast-def}
{\mathfrak{F}^{\ast}}:=\mathfrak{F}-\sum_{i=1}^2 a_{ii}\partial_{ii}\varphi_{\rm en}^{\ast},
\end{equation}
where $a_{ii}$ ($i=1,2$) are given by \eqref{2D-aij}.
By the standard elliptic theory, we know that the linear boundary problem \eqref{lin-phi-hom} has a unique solution $\phi_{\rm hom}\in C^{1,\alpha}(\overline{\mathcal{N}_{L,f_{\ast}}^+})\cap C^{2,\alpha}(\mathcal{N}_{L,f_{\ast}}^+)$. 

To obtain a uniform $C^0$ estimate of $\phi_{\rm hom}$ for all $L$, we define a function $\mathfrak{M}$ by 
$$\mathfrak{M}(x_1,x_2):=-\left(\frac{\|\mathfrak{F}^{\ast}\|_{0,\alpha,\mathcal{N}_{L,f_\ast}^+}}{a_{22}}\right)x_2^2+\left(\frac{3\|\mathfrak{F}^{\ast}\|_{0,\alpha,\mathcal{N}_{L,f_\ast}^+}}{a_{22}}\right)x_2+\frac{5\|\mathfrak{F}^{\ast}\|_{0,\alpha,\mathcal{N}_{L,f_\ast}^+}}{a_{22}}.$$
Since $a_{22}>\nu>0$ in $\mathcal{N}_{L,f_{\ast}}^+$ by \eqref{nu}, $\mathfrak{M}$ is well-defined.
And, by a direct computation, one can easily check that $\mathfrak{M}\pm\phi_{\rm hom}$ satisfies
\begin{equation*}
\left\{\begin{split}
\mathcal{L}(\mathfrak{M}\pm\phi_{\rm hom})=-2\|\mathfrak{F}^{\ast}\|_{0,\alpha,\mathcal{N}_{L,f_\ast}^+}\pm \mathcal{L}(\phi_{\rm hom})\le0\quad&\mbox{in}\quad\mathcal{N}_{L,f_\ast}^+,\\
\mathfrak{M}\pm\phi_{\rm hom}\ge0\quad&\mbox{on}\quad\Gamma_{\rm en}^+,\\
\mathfrak{M}\pm\phi_{\rm hom}=\mathfrak{M}\ge0\quad&\mbox{on}\quad\Gamma_{\rm cd}^{L,f_{\ast}}\cup\Gamma^{L,f_{\ast}}_{\rm ex},\\
\partial_{x_2}(\mathfrak{M}\pm\phi_{\rm hom})=\partial_{x_2}\mathfrak{M}\ge0\quad&\mbox{on}\quad\Gamma_{\rm w}^{+,L}.
\end{split}\right.
\end{equation*}
Since $\mathcal{L}$ is uniformly elliptic, the comparison principle and Hopf's lemma imply
$-\mathfrak{M}\le\phi_{\rm hom}\le\mathfrak{M}$ in $\mathcal{N}_{L,f_\ast}^+$.
Therefore we get 
$$\|\phi_{\rm hom}\|_{0,\mathcal{N}_{L,f_\ast}^+}\le C\|\mathfrak{F}^{\ast}\|_{\alpha,\mathcal{N}_{L,f_\ast}^+}.$$

To obtain $C^{2,\alpha}$ estimate of $\phi$ up to the boundary, 
we use the method of reflection.
By the compatibility conditions of $(S_\ast,\tilde\phi)$ given in \eqref{S-set} and \eqref{phi-set}, and $\partial_{x_1}\psi\equiv0$ on $\Gamma_{\rm en}^{\epsilon}\cup\Gamma^{L,f_\ast}_{\rm ex}$ given from \eqref{lin-psi}, we have 
\begin{equation}\label{F-0-ex}
\mathfrak{F}=\mbox{div}{\bf F}(S_{\ast}-S_0^+,\nabla\tilde{\phi},\nabla{\psi})\equiv0\quad\mbox{on}\quad \Gamma_{\rm en}^{\epsilon}\cup\Gamma^{L,f_{\ast}}_{\rm ex}.
\end{equation}
From the definition of $\varphi_{\rm en}^{\ast}$ given in \eqref{def-var-en}, the compatibility condition of $f_\ast$ given in \eqref{f-set}, and the definition of $\eta$ given in \eqref{def-eta}, it can be directly checked that 
\begin{equation}\label{varphi-ex}
\partial_{x_1x_1}\varphi_{\rm en}^{\ast}=\partial_{x_2x_2}\varphi_{\rm en}^{\ast}\equiv0\quad\mbox{on}\quad \Gamma_{\rm en}^{\epsilon}\cup\Gamma^{L,f_{\ast}}_{\rm ex}.
\end{equation}
By \eqref{F-0-ex}, \eqref{varphi-ex}, and the definition of $\mathfrak{F}^{\ast}$ given in \eqref{Fast-def}, we have 
${\mathfrak{F}^{\ast}}\equiv 0$ on $\Gamma_{\rm en}^\epsilon\cup\Gamma^{L,f_{\ast}}_{\rm ex}$.
By using this compatibility condition of $\mathfrak{F}^{\ast}$ and the method of reflection, we get 
$$\|\phi_{\rm hom}\|_{2,\alpha,\mathcal{N}_{L,f_{\ast}}^+}\le C\|\mathfrak{F}^{\ast}\|_{\alpha,\mathcal{N}_{L,f_{\ast}}^+}.$$
Therefore the linear boundary value problem \eqref{lin-phi} has a unique solution $\phi=\phi_{\rm hom}+{\varphi_{\rm en}^{\ast}}\in C^{2,\alpha}(\overline{\mathcal{N}_{L,f_{\ast}}^+})$, and $\phi$ satisfies
\begin{equation*}
\begin{split}
\|\phi\|_{2,\alpha,\mathcal{N}_{L,f_{\ast}}^+}
&\le C\left(\|\mathfrak{F}\|_{\alpha,\mathcal{N}_{L,f_{\ast}}^+}+\|\varphi_{\rm en}\|_{2,\alpha,\Gamma_{\rm en}^+}\right).
\end{split}
\end{equation*}

It follows from \eqref{lin-phi-hom} and \eqref{varphi-ex} that
\begin{equation}\label{phi-0-ex}
\partial_{x_2x_2}\phi\equiv0 \quad\mbox{on}\quad\Gamma_{\rm en}^{\epsilon}\cup\Gamma^{L,f_{\ast}}_{\rm ex}.
\end{equation}
By \eqref{F-0-ex} and \eqref{phi-0-ex}, we have 
$\mathcal{L}(\phi)=a_{11}\partial_{x_1x_1}\phi\equiv0$ on $\Gamma_{\rm en}^{\epsilon}\cup\Gamma^{L,f_{\ast}}_{\rm ex}.$
Since $a_{11}>0$, we conclude that 
$\partial_{x_1x_1}\phi=0$ on $\Gamma_{\rm en}^{\epsilon}\cup\Gamma^{L,f_{\ast}}_{\rm ex}.$
The proof of Claim is completed.
\end{proof}

{\bf 5.} For fixed $(S_\ast,f_\ast)\in\mathcal{S}(M_1)\times\mathcal{F}(M_2)$, define an iteration mapping $\mathcal{I}^{f_{\ast},S_{\ast}}:\mathcal{K}^{f_\ast}(M_3,M_4)\rightarrow \left[C^{2,\alpha}(\overline{\mathcal{N}_{L,f_\ast}^+})\right]^2$ by 
\begin{equation*}
\mathcal{I}^{f_\ast,S_\ast}(\tilde{\phi},\tilde{\psi})=(\phi,\psi),
\end{equation*}
where $(\phi,\psi)$ is the solution to \eqref{lin-psi} and \eqref{lin-phi} associated with $(\tilde{\phi},\tilde{\psi})$.

By a direct computation, one can easily check that there exists a constant $\epsilon_1\in(0,\frac{1}{4})$ depending only on the data so that if 
\begin{equation}\label{def-ep1}
M_1\sigma+M_2\sigma+M_3\sigma+M_4\sigma\le\epsilon_1,
\end{equation}
 then we have 
\begin{equation}\label{FAG}
\left.\begin{split}
&\|\mathfrak{F}\|_{\alpha,\mathcal{N}_{L,f_{\ast}}^+}\le C\left(M_1\sigma+(M_3\sigma)^2+M_4\sigma\right),\\
&\|\mathfrak{B}\|_{1,\alpha,\Gamma_{\rm cd}^{L,f_{\ast}}}\le C\left(M_1\sigma+(M_2\sigma)^2\right),\\
&\|\mathfrak{G}\|_{\alpha,\mathcal{N}_{L,f_{\ast}}^+}\le C M_1\sigma,
\end{split}\right.
\end{equation}
where $\mathfrak{F}$, $\mathfrak{B}$, and $\mathfrak{G}$ are given by \eqref{FF_def} and \eqref{GB-def}.
Then it follows from \eqref{psi-est-2D}, \eqref{phi-est-2D}, and \eqref{FAG} that
\begin{equation}\label{2D-pp-est1}
\left.\begin{split}
&\|\psi\|_{2,\alpha,\mathcal{N}_{L,f_{\ast}}^+}\le C_1^{\flat}\left(M_1\sigma+(M_2\sigma)^2\right),\\
&\|\phi\|_{2,\alpha,\mathcal{N}_{L,f_{\ast}}^+}\le C_1^{\flat}\left(M_1\sigma+(M_3\sigma)^2+M_4\sigma+\sigma\right),
\end{split}\right.
\end{equation}
where the constant $C_1^{\flat}>0$ depends only on the data and $\alpha$ but independent of $L$.
Choose constants $M_3$, $M_4$, and $\sigma_6^{\ast}$ as
\begin{equation}\label{2D-sigma8}
\begin{split}
&M_3=\max\left\{4C_1^{\flat},4C_1^{\flat}M_1,4C_1^{\flat}M_4\right\},\quad M_4=2C_1^{\flat}M_1,\\
&\mbox{and}\quad\sigma_6^{\ast}=\min\left\{\frac{\epsilon_1}{M_1+M_2+M_3+M_4},\frac{1}{4C_1^{\flat}M_3},\frac{M_4}{2C_1^{\flat}M_2^2}\right\},
\end{split}
\end{equation}
where $\epsilon_1$ is given by \eqref{def-ep1},
so that \eqref{2D-pp-est1} implies that 
$(\phi,\psi)\in\mathcal{K}^{f_\ast}(M_3,M_4)$ for $\sigma\le\sigma_6^{\ast}$.
Under such choices of $(M_3,M_4,\sigma_6^{\ast})$, the iteration mapping $\mathcal{I}^{f_\ast,S_\ast}$ maps $\mathcal{K}^{f_\ast}(M_3,M_4)$ into itself if  $\sigma\le\sigma_6^{\ast}$.
Furthermore, $(\phi,\psi)$ satisfies the estimate
\begin{equation*}
\|\phi\|_{2,\alpha,\mathcal{N}_{L,f_{\ast}}^+}+\|\psi\|_{2,\alpha,\mathcal{N}_{L,f_{\ast}}^+}\le (M_3+M_4)\sigma\le C(M_1+1)\sigma.
\end{equation*}

 Now we show that $\mathcal{I}^{f_\ast,S_{\ast}}$ is a contraction mapping if $\sigma$ is small.
 
For each $j=1,2$, set
\begin{equation*}
\left\{\begin{split}
&(\phi^{(j)},\psi^{(j)}):=\mathcal{I}^{f_\ast,S_\ast}(\tilde{\phi}^{(j)},\tilde{\psi}^{(j)})\quad\mbox{for}\quad (\tilde{\phi}^{(j)},\tilde{\psi}^{(j)})\in\mathcal{K}^{f_\ast}(M_3,M_4),\\
&{\bf F}^{(j)}:=\mbox{div}{\bf F}(S_{\ast}-S_0^+,\nabla\tilde{\phi}^{(j)},\nabla\psi^{(j)}),\\
&G^{(j)}:=G(S_{\ast},\partial_{x_2}S_{\ast},\nabla\tilde{\phi}^{(j)}+\nabla\varphi_0,\nabla\tilde{\psi}^{(j)}),
\end{split}\right.
\end{equation*}
where ${\bf F}$ and $G$ are given by \eqref{2D-def-F} and \eqref{2D_FGH}, respectively. 
By a direct computation, it can be checked that there exists a constant $\epsilon_2\in(0,\epsilon_1]$ depending only on the data so that if 
$$M_1\sigma+M_2\sigma+M_3\sigma+M_4\sigma\le \epsilon_2,$$
then we have 
\begin{equation}\label{FG-diff}
\left.\begin{split}
&\|{\bf F}^{(1)}-{\bf F}^{(2)}\|_{\alpha,\mathcal{N}_{L,f_\ast}^+}\le C(M_1+1)\sigma\|\tilde{\phi}^{(1)}-\tilde{\phi}^{(2)}\|_{2,\alpha,\mathcal{N}_{L,f_\ast}^+}+C\|\psi^{(1)}-\psi^{(2)}\|_{2,\alpha,\mathcal{N}_{L,f_\ast}^+},\\
&\|G^{(1)}-G^{(2)}\|_{\alpha,\mathcal{N}_{L,f_\ast}^+}\le CM_1\sigma\left(\|\tilde{\phi}^{(1)}-\tilde{\phi}^{(2)}\|_{2,\alpha,\mathcal{N}_{L,f_\ast}^+}+\|\tilde{\psi}^{(1)}-\tilde{\psi}^{(2)}\|_{2,\alpha,\mathcal{N}_{L,f_\ast}^+}\right).
\end{split}\right.
\end{equation}
Then, by \eqref{psi-est-2D}, \eqref{phi-est-2D}, and \eqref{FG-diff}, we have
\begin{equation*}
\begin{split}
\|\phi^{(1)}-&\phi^{(2)}\|_{2,\alpha,\mathcal{N}_{L,f_{\ast}}^+}+\|\psi^{(1)}-\psi^{(2)}\|_{2,\alpha,\mathcal{N}_{L,f_{\ast}}^+}\\
&\le C_2^{\flat}(M_1+1)\sigma\left(\|\tilde{\phi}^{(1)}-\tilde{\phi}^{(2)}\|_{2,\alpha,\mathcal{N}_{L,f_{\ast}}^+}+\|\tilde{\psi}^{(1)}-\tilde{\psi}^{(2)}\|_{2,\alpha,\mathcal{N}_{L,f_{\ast}}^+}\right)
\end{split}
\end{equation*}
for a constant $C_2^{\flat}>0$ depending only on the data and $\alpha$ but independent of $L$. 
Choose $\sigma_6$ as 
\begin{equation}\label{Sigma6}
\sigma_6=\min\left\{\sigma_6^{\ast},\frac{\epsilon_2}{M_1+M_2+M_3+M_4},\frac{1}{2C_2^{\flat}(M_1+1)}\right\}
\end{equation}
for $\sigma_6^{\ast}$ defined in \eqref{2D-sigma8}. Thus if $\sigma\le\sigma_6$, then the mapping $\mathcal{I}^{f_\ast,S_\ast}$ is a contraction mapping so that $\mathcal{I}^{f_\ast,S_\ast}$ has a unique fixed point in $\mathcal{K}^{f_\ast}(M_3,M_4)$.
This gives the unique existence of a solution to \eqref{Ite-3D}.
The proof of Lemma \ref{2D-fix} is completed.
\end{proof}

Now we are ready to prove the unique solvability of Problem \ref{FBP-S}.

\begin{proof}[Proof of Lemma \ref{S-free-2}]
{\bf 1.} 
Let $(\varphi,\psi)\in\left[C^{2,\alpha}(\overline{\mathcal{N}_{L,f_\ast}^+})\right]^2$ be a solution of the fixed boundary value problem \eqref{Ite-3D} associated with $S_{\ast}\in\mathcal{S}(M_1)$ 
that satisfies the estimate \eqref{2D-Fix-est} given  in Lemma \ref{2D-fix}.
For simplicity, we set 
\begin{equation*}
\rho^{\ast}:=H(S_{\ast},\nabla\varphi,\nabla\psi),
\quad (u_1^{\ast},u_2^{\ast}):=\nabla\varphi+\nabla^{\bot}\psi,
\end{equation*} 
where $H$ is given by \eqref{2D_FGH}.
As in the proof of Lemma \ref{2D-fix}, there exists a constant $\epsilon_3>0$ depending only on the data so that if 
\begin{equation}\label{epsilon3}
M_1\sigma+M_2\sigma+M_3\sigma+M_4\sigma\le \epsilon_3,
\end{equation}
then we have 
\begin{equation*}
\|\rho^{\ast}(u_1^{\ast},u_2^{\ast})-\rho_0^+(u_0,0)\|_{1,\alpha,\mathcal{N}_{L,f_\ast}^+}\le C_{\star}\left(M_1+1\right)\sigma,
\end{equation*}
where the constant $C_{\star}>0$ depends only on the data and $\alpha$ but independent of $L$.
If it holds that 
$$\sigma\le \frac{\rho_0^+u_0}{2C_{\star}\left(M_1+1\right)},$$
then we have
\begin{equation*}\label{rhouest}
\|\rho^{\ast}(u_1^{\ast},u_2^{\ast})-\rho_0^+(u_0,0)\|_{1,\alpha,\mathcal{N}_{L,f_\ast}^+}\le  \frac{\rho_0^+u_0}{2}.
\end{equation*}

For each $x_1\in[0,L]$, we choose $f(x_1)\in\mathbb{R}$ to satisfy 
\begin{equation}\label{New-f}
\int_{f(x_1)}^{f_\ast(x_1)}\rho_0^+u_0dt
=\int_0^1\rho^{\ast}u_1^{\ast}(0,t)dt-\int_{f_{\ast}(x_1)}^1\rho^{\ast}u_1^{\ast}(x_1,t)dt.
\end{equation}
If $f\equiv f_\ast$, then 
\begin{equation}\label{fix-f}
\int_0^1\rho^{\ast}u_1^{\ast}(0,t)dt=\int_{f(x_1)}^1\rho^{\ast}u_1^{\ast}(x_1,t)dt\quad\mbox{for all}\quad 0\le x_1\le L.
\end{equation} 
Differentiating the equation in \eqref{fix-f} with respect to $x_1$ and using the equation $\partial_{x_1}(\rho^{\ast}u_1^{\ast})+\partial_{x_2}(\rho^{\ast}u_2^{\ast})=0$, we have
\begin{equation*}
f'(x_1)=\frac{u_2^{\ast}}{u_1^{\ast}}(x_1,f(x_1))=\frac{(\nabla\varphi+\nabla^{\bot}\psi)\cdot{\bf e}_2}{(\nabla\varphi+\nabla^{\bot}\psi)\cdot{\bf e}_1}(x_1,f(x_1)).
\end{equation*}
Also, we have $f(0)=0$. 
Thus $f$ satisfies the free boundary condition \eqref{2D_BC_Fin2} for $0<x_1<L$. 
Since $\rho_0^+u_0>0$, \eqref{New-f} directly yields that 
\begin{equation}\label{2D-free-BC}
f(x_1)=f_\ast(x_1)-\frac{1}{\rho_0^+u_0}\int_0^1\rho^{\ast}u_1^{\ast}(0,t)dt+\frac{1}{\rho_0^+u_0}\int_{f_{\ast}(x_1)}^1\rho^{\ast}u_1^{\ast}(x_1,t)dt.
\end{equation}
By a direct computation with using \eqref{2D-free-BC}, we obtain the compatibility conditions $f(0)=f'(0)=f'(L)=0$ and 
the estimate
\begin{equation}\label{f-est-c2}
\|f\|_{2,\alpha,(0,L)}\le C_{\star\star}(M_1+1)\sigma
\end{equation}
for a constant $C_{\star\star}>0$ depending only on the data and $\alpha$ but independent of $L$.

We define an iteration mapping $\mathcal{I}^{S_\ast}:\mathcal{F}(M_2)\rightarrow C^{2,\alpha}([0,L])$  by
$$\mathcal{I}^{S_{\ast}}(f_{\ast})=f$$
for $f$ given by \eqref{2D-free-BC}.
Choose $M_2$ and $\sigma_5^{\ast}$ as 
\begin{equation}\label{2D-sigma4}
\begin{split}
&M_2=C_{\star\star}(M_1+1)\\
&\mbox{and}\quad\sigma_5^{\ast}=\min\left\{\sigma_6,\frac{\rho_0^+u_0}{2C_{\star}\left(M_1+1\right)},\frac{\epsilon_3}{M_1+M_2+M_3+M_4}\right\}
\end{split}
\end{equation}
with $\sigma_6$ defined in \eqref{Sigma6} and $\epsilon_3$ given in \eqref{epsilon3}.
Then the mapping $\mathcal{I}^{S_\ast}$ maps $\mathcal{F}(M_2)$ into itself if $\sigma\le \sigma_5^{\ast}$.

{\bf 2.} 
The iteration set $\mathcal{F}(M_2)$ given by \eqref{f-set}  is a convex and compact subset of $C^{2,\alpha/2}([0,L])$.
For each fixed $S_{\ast}\in\mathcal{S}(M_1)$, the iteration map $\mathcal{I}^{S_\ast}$ maps $\mathcal{F}(M_2)$ into itself where $M_2$ is chosen by \eqref{2D-sigma4}, and $\sigma\le\sigma_5^{\ast}$ for $\sigma_5^{\ast}$ from \eqref{2D-sigma4}.

Suppose that a sequence $\{f_{\ast}^{(k)}\}_{k=1}^{\infty}\subset\mathcal{F}(M_2)$ converges in $C^{2,\alpha/2}([0,L])$ to $f_{\ast}^{(\infty)}\in\mathcal{F}(M_2)$. 
For each $k\in\mathbb{N}\cup\{\infty\}$, set
\begin{equation*}
f^{(k)}:=\mathcal{I}^{S_{\ast}}(f_{\ast}^{(k)}).
\end{equation*}
And, let $\mathcal{M}^{(k)}:=(\varphi^{(k)},\psi^{(k)})\in\left[C^{2,\alpha}(\overline{\mathcal{N}_{L,f_\ast^{(k)}}^+})\right]^2$ be the unique solution of \eqref{Ite-3D} associated with $f_{\ast}=f_{\ast}^{(k)}$.
Define a transformation $T^{(k)}:\overline{\mathcal{N}_{L,f_{\ast}^{(\infty)}}^+}\rightarrow\overline{\mathcal{N}_{L,f_{\ast}^{(k)}}^+}$ by 
\begin{equation*}
T^{(k)}(x_1,x_2)=\left(x_1,\frac{1-f_\ast^{(k)}(x_1)}{1-f_{\ast}^{(\infty)}(x_1)}(x_2-1)+1\right)\quad\mbox{for}\quad (x_1,x_2)\in\overline{\mathcal{N}_{L,f_{\ast}^{(\infty)}}^+}.
\end{equation*}
Then $\{\mathcal{M}^{(k)}\circ T^{(k)}\}_{k=1}^{\infty}$ is sequentially compact in $\left[C^{2,\alpha/2}(\overline{\mathcal{N}_{L,f_\ast^{(\infty)}}^+})\right]^2$ and the limit of each convergent subsequence of $\{\mathcal{M}^{(k)}\circ T^{(k)}\}_{k=1}^{\infty}$ in $ \left[C^{2,\alpha/2}(\overline{\mathcal{N}_{L,f_\ast^{(\infty)}}^+})\right]^2$ solves \eqref{Ite-3D} associated with $f_{\ast}=f_{\ast}^{(\infty)}$. 
By the uniqueness of a solution for the problem \eqref{Ite-3D},  $\{\mathcal{M}^{(k)}\circ T^{(k)}\}_{k=1}^{\infty}$ is convergent in $\left[C^{2,\alpha/2}(\overline{\mathcal{N}_{L,f_\ast^{(\infty)}}^+})\right]^2$. 
It follows from \eqref{2D-free-BC} and \eqref{f-est-c2} that $f^{(k)}$ converges to $f^{(\infty)}$ in $C^{2,\alpha/2}([0,L])$.
This implies that $\mathcal{I}^{S_\ast}(f_{\ast}^{(k)})$ converges to $\mathcal{I}^{S_\ast}(f_{\ast}^{(\infty)})$ in $C^{2,\alpha/2}([0,L])$.
Thus $\mathcal{I}^{S_\ast}$ is continuous in $C^{2,\alpha/2}([0,L])$.
Applying the Schauder fixed point theorem yields that $\mathcal{I}^{S_\ast}$ has a fixed point $f\in\mathcal{F}(M_2)$.
For such $f$, let $(\varphi,\psi)\in\left[C^{2,\alpha}(\overline{\mathcal{N}_{L,f}^+})\right]^2$ be the unique solution to the fixed boundary problem \eqref{Ite-3D} associated with $f_{\ast}=f$. 
Then $(f,\varphi,\psi)$ is a solution to Problem \ref{FBP-S}.
It follows from  \eqref{2D-Fix-est} and \eqref{f-est-c2} that $(f,\varphi,\psi)$ satisfies the estimate
\begin{equation*}
\|f\|_{2,\alpha,(0,L)}+\|\varphi-\varphi_0\|_{2,\alpha,\mathcal{N}_{L,f_{\ast}}^+}+\|\psi\|_{2,\alpha,\mathcal{N}_{L,f_{\ast}}^+}\le  C(M_1+1)\sigma.
\end{equation*}

{\bf 3.}
Finally, it remains to prove the uniqueness of a solution to Problem \ref{FBP-S}.
For a fixed $S_{\ast}\in\mathcal{S}(M_1)$, let 
$(f^{(1)},\varphi^{(1)},\psi^{(1)})$ and $(f^{(2)},\varphi^{(2)},\psi^{(2)})$ be two solutions to Problem \ref{FBP-S}, and suppose that each solution satisfies the estimate given in \eqref{S-free-est} of Lemma \ref{S-free-2}.
Define a transformation $\mathfrak{T}:\overline{\mathcal{N}_{L,f^{(1)}}^+}\rightarrow\overline{\mathcal{N}_{L,f^{(2)}}^+}$ by
\begin{equation}\label{Trans-12}
\mathfrak{T}(x_1,x_2)=\left(x_1,\frac{(1-f^{(2)}(x_1))x_2+(f^{(2)}(x_1)-f^{(1)}(x_1))}{1-f^{(1)}(x_1)}\right).
\end{equation}
Since $1-f^{(1)}>0$, the transformation $\mathfrak{T}$ is invertible and 
\begin{equation*}
\mathfrak{T}^{-1}(y_1,y_2)=\left(y_1,\frac{(1-f^{(1)}(y_1))y_2+(f^{(1)}(y_1)-f^{(2)}(y_1))}{1-f^{(2)}(y_1)}\right).
\end{equation*}
Set
\begin{equation*}
\left\{\begin{split}
&\widetilde{\phi}:=\varphi^{(1)}-\left(\varphi^{(2)}\circ\mathfrak{T}\right),\quad \widetilde{\psi}:=\psi^{(1)}-\left(\psi^{(2)}\circ\mathfrak{T}\right),\\
&\widetilde{S}:=S_{\ast}-\left(S_{\ast}\circ\mathfrak{T}\right),\quad \widetilde{f}:=f^{(1)}-f^{(2)}.
\end{split}\right.
\end{equation*}
We first rewrite the nonlinear boundary value problem \eqref{S-free} for $(\varphi^{(2)},\psi^{(2)})$ in $\mathcal{N}_{L,f^{(2)}}^+$ as a nonlinear boundary value problem for $(\varphi^{(2)}\circ\mathfrak{T},\psi^{(2)}\circ\mathfrak{T})$ in $\mathcal{N}_{L,f^{(1)}}^+$, and subtract the resultant equations and boundary conditions from the nonlinear boundary value problem \eqref{S-free} for $(\varphi^{(1)},\psi^{(1)})$ in $\mathcal{N}_{L,f^{(1)}}^+$.
Then we get a nonlinear boundary value problem for $(\widetilde{\phi},\widetilde{\psi})$ in $\mathcal{N}_{L,f^{(1)}}^+$.
By adjusting the proof of Lemma \ref{2D-fix} with using the estimate 
\begin{equation*}
\|\widetilde{S}\|_{1,\alpha,\mathcal{N}_{L,f^{(1)}}^+}\le CM_1\sigma\|\widetilde{f}\|_{2,\alpha,(0,L)},
\end{equation*} we obtain
\begin{equation*}
\begin{split}
\|\widetilde{\phi}\|_{2,\alpha,\mathcal{N}_{L,f^{(1)}}^+}+\|\widetilde{\psi}\|_{2,\alpha,\mathcal{N}_{L,f^{(1)}}^+}
\le& C_1^{\ast}(M_1+1)\sigma\left(\|\widetilde{\phi}\|_{2,\alpha,\mathcal{N}_{L,f^{(1)}}^+}+\|\widetilde{\psi}\|_{2,\alpha,\mathcal{N}_{L,f^{(1)}}^+}\right)\\
&+C(M_1+1)\sigma\|\widetilde{f}\|_{2,\alpha,(0,L)},
\end{split}
\end{equation*}
where the constant $ C_1^{\ast}>0$ depends only on the data and $\alpha$ but independent of $L$.
If it holds that 
$$\sigma\le \frac{1}{2C_1^{\ast}(M_1+1)},$$
then we obtain from the previous estimate that
\begin{equation}\label{R-g}
\|\widetilde{\phi}\|_{2,\alpha,\mathcal{N}_{L,f^{(1)}}^+}+\|\widetilde{\psi}\|_{2,\alpha,\mathcal{N}_{L,f^{(1)}}^+}\le C(M_1+1)\sigma\|\widetilde{f}\|_{2,\alpha,(0,L)}.
\end{equation}
By using the free boundary condition \eqref{2D_BC_Fin2}, we can express $\widetilde{f}'$ in terms of $(\widetilde{\phi},\widetilde{\psi},\mathfrak{T},D_{(x_1,x_2)}\mathfrak{T})$.
Then we apply \eqref{R-g} to obtain the estimate 
\begin{equation}\label{f-1-alpha}
\begin{split}
\|{\widetilde{f}}'\|_{1,\alpha,(0,L)}&\le C\left(\|\widetilde{\phi}\|_{2,\alpha,\mathcal{N}_{L,f^{(1)}}^+}+\|\widetilde{\psi}\|_{2,\alpha,\mathcal{N}_{L,f^{(1)}}^+}\right)+C(M_1+1)\sigma\|\widetilde{f}\|_{2,\alpha,(0,L)}\\
&\le C(M_1+1)\sigma\|\widetilde{f}\|_{2,\alpha,(0,L)}.
\end{split}
\end{equation}

Now we  estimate $\|\widetilde{f}\|_{0,(0,L)}$.
Define $\rho^{(1)}$, $u_1^{(1)}$, $\rho^{(2)}$, and $u_1^{(2)}$ by 
\begin{equation*}
\begin{split}
&\rho^{(1)}:=H(S_{\ast},\nabla\varphi^{(1)},\nabla\psi^{(1)}),\quad u_1^{(1)}:=\partial_{x_1}\varphi^{(1)}+\partial_{x_2}\psi^{(1)},\\
&\rho^{(2)}:=H(S_{\ast},D\varphi^{(2)},D\psi^{(2)}),\quad u_1^{(2)}:=\partial_{y_1}\varphi^{(2)}+\partial_{y_2}\psi^{(2)},
\end{split}
\end{equation*}
where $H$ is given in \eqref{2D_FGH}, $D=(\partial_{y_1},\partial_{y_2})$.
By using \eqref{fix-f}, we get 
\begin{equation}\label{F12}
\begin{split}
\int_0^1&\left(\rho^{(1)}u_1^{(1)}-\rho^{(2)}u_1^{(2)}\right)(0,t)dt\\
&=\int_{f^{(1)}(x_1)}^1\rho^{(1)}u_1^{(1)}(x_1,t)dt-\int_{f^{(2)}(x_1)}^1\rho^{(2)}u_1^{(2)}(x_1,t)dt.
\end{split}
\end{equation}
Fix $x_0\in[0,L]$. 
Without loss of generality, we may assume that 
$$f^{(1)}(x_0)\ge f^{(2)}(x_0).$$
Then \eqref{F12} can be rewritten as 
\begin{equation*}\label{F12-re}
\begin{split}
\int_0^1&\left(\rho^{(1)}u_1^{(1)}-\rho^{(2)}u_1^{(2)}\right)(0,t)dt\\
&=\int_{f^{(1)}(x_0)}^1\left(\rho^{(1)}u_1^{(1)}-\rho^{(2)}u_1^{(2)}\right)(x_0,t)dt-\int_{f^{(2)}(x_0)}^{f^{(1)}(x_0)}\rho^{(2)}u_1^{(2)}(x_0,t)dt.
\end{split}
\end{equation*}
By applying \eqref{R-g}, we have
\begin{equation*}
0\le f^{(1)}(x_0)-f^{(2)}(x_0)\le C(M_1+1)\sigma\|\widetilde{f}\|_{2,\alpha,(0,L)}.
\end{equation*}
Combining this with \eqref{f-1-alpha}, we finally get 
\begin{equation}\label{Final-f-est}
\|\widetilde{f}\|_{2,\alpha,(0,L)}\le C_2^{\ast}(M_1+1)\sigma\|\widetilde{f}\|_{2,\alpha,(0,L)}.
\end{equation}
where the constant $C_2^{\ast}>0$ depends only on the data and $\alpha$ but independent of $L$.
Choose $\sigma_5$ as
\begin{equation}\label{sigma5}
\sigma_5=\min\left\{\sigma_5^{\ast},\frac{1}{2C_1^{\ast}(M_1+1)},\frac{1}{2C_2^{\ast}(M_1+1)}\right\}
\end{equation}
for $\sigma_5^{\ast}$ defined in \eqref{2D-sigma4}.
Then \eqref{Final-f-est} implies that $f^{(1)}=f^{(2)}$. 
By Lemma \ref{2D-fix}, $(\varphi^{(1)},\psi^{(1)})=(\varphi^{(2)},\psi^{(2)}).$
The proof of Lemma \ref{S-free-2} is completed.
\end{proof}

\subsection{Proof of Proposition \ref{3D-Prop}}\label{Prop-sec}


For a fixed $S_{\ast}\in\mathcal{S}(M_1)$, let $(f,\varphi,\psi)\in C^{2,\alpha}([0,L])\times\left[C^{2,\alpha}(\overline{\mathcal{N}_{L,f}^+})\right]^2$ be a solution to Problem \ref{FBP-S}. 
By Lemma \ref{S-free-2}, if $\sigma\le \sigma_5$ for $\sigma_5$ given in \eqref{sigma5}, then there exists a unique solution $(f,\varphi,\psi)$ that satisfies the estimate \eqref{S-free-est}.

\begin{lemma}\label{2D-trans}
Under the same assumptions on $(S_{\rm en}, v_{\rm en})$ as in {\emph{Proposition \ref{3D-Prop}}}, 
there exists a small constant $\sigma_4^{\ast\ast}\in(0,\sigma_5]$ depending only on the data and $\alpha$ so that if 
$$\sigma=\|S_{en}-S_0\|_{2,\alpha,\Gamma_{\rm en}^+}+\|v_{\rm en}\|_{1,\alpha,\Gamma_{\rm en}^+}\le\sigma_4^{\ast\ast},$$ then 
the initial value problem \eqref{Ite-3D2} has a unique solution $S\in C^{2,\alpha}(\overline{\mathcal{N}_{L,f}^+})$ satisfying
\begin{equation}\label{2D-def-C}
\|S-S_0^+\|_{2,\alpha,\mathcal{N}_{L,f}^+}\le C^{\ast}\|S_{\rm en}-S_0\|_{2,\alpha,\Gamma_{\rm en}^+},
\end{equation}
where the constant $C^{\ast}>0$ depends only on the data and $\alpha$ but independent of $L$.
\end{lemma}

As in \cite{B-D-X}, we can obtain a solution 
\begin{equation}\label{2D-S-EN}
S({\bf x})=S_{\rm en}\circ\mathcal{Y}_0({\bf x}),
\end{equation}
where the function $\mathcal{Y}_0$ is defined by
\begin{equation}\label{y0-def}
\mathcal{Y}_0({\bf x}):=\mathcal{G}^{-1}\circ w({\bf x})\quad\mbox{for}\quad {\bf x}\in\overline{\mathcal{N}_{L,f}^+}.
\end{equation}
Here, $\mathcal{G}:[0,1]\rightarrow [w(0,0),w(0,1)]$ is an invertible function such that
\begin{equation}\label{mathcalG}
\mathcal{G}(x_2)=w(0,x_2)\quad\mbox{for}\quad x_2\in[0,1]
\end{equation}
and
$w$ is a function defined by 
\begin{equation*}\label{def-w}
w(x_1,x_2):=\int_{1}^{x_2}{\bf V}\cdot{\bf e}_1(x_1,t)dt\quad\mbox{for}\quad (x_1,x_2)\in\overline{\mathcal{N}_{L,f}^+}
\end{equation*}
for
${\bf V}:=H(S_{\ast},\nabla\varphi,\nabla\psi)(\nabla\varphi+\nabla^{\bot}\psi)$, where $H$ is given by \eqref{2D_FGH}.
Moreover, $\mathcal{Y}_0$ satisfies
\begin{equation*}\label{2D-nablaY0}
\|\mathcal{Y}_0\|_{2,\alpha,\mathcal{N}_{L,f}^+}\le C\|{\bf V}\|_{1,\alpha,\mathcal{N}_{L,f}^+}.
\end{equation*}
Thus we have
\begin{equation*}
\|S-S_0^+\|_{2,\alpha,\mathcal{N}_{L,f}^+}=\|S_{\rm en}\circ\mathcal{Y}_0-S_0^+\|_{2,\alpha,\mathcal{N}_{L,f}^+}\le C^{\ast}\|S_{\rm en}-S_0\|_{2,\alpha,\Gamma_{\rm en}^+}
\end{equation*}
for a constant $C^{\ast}>0$ depending only on the data and $\alpha$ but independent of $L$.

For $\mathcal{N}_L$ given by \eqref{def-NL} and $\mathcal{N}_{L,2f-1}^+:=\mathcal{N}_L\cap\{x_2>2f(x_1)-1\}$,
consider a transformation $\mathfrak{P}_f:\mathcal{N}_{L,2f-1}^+\rightarrow\mathcal{N}_L$ defined by 
\begin{equation*}
\mathfrak{P}_f(x_1,x_2)=\left(x_1,\frac{x_2-f(x_1)}{1-f(x_1)}\right)\quad\mbox{for}\quad (x_1,x_2)\in\mathcal{N}_{L,2f-1}^+.
\end{equation*}
Then $\mathfrak{P}_f$ is invertible and 
\begin{equation*}
\mathfrak{P}_f^{-1}(y_1,y_2)=(y_1,(1-y_2)f(y_1)+y_2)\quad\mbox{for}\quad (y_1,y_2)\in\mathcal{N}_L.
\end{equation*}
For the unique solution $S$ of the initial-value problem \eqref{Ite-3D2}, 
define a function $S^e$ by
\begin{equation*}
S^e(y_1,y_2):=\sum_{i=1}^3 c_i\left(S\circ\mathfrak{P}_f^{-1}\right)\left(y_1,-\frac{y_2}{i}\right)\quad \mbox{for}\quad -1<y_2<0,
\end{equation*}
where $c_1=6$, $c_2=-32$, and $c_3=27$, which are determined by the system of equations 
\begin{equation*}
\begin{split}
&\sum_{i=1}^3 c_i\left(-\frac{1}{i}\right)^m=1,\quad m=0,1,2.\\
\end{split}
\end{equation*}
With such $S^e$, define an extension of $S$ into $\mathcal{N}_{L,-1/2}^+$ as follows:
\begin{equation}\label{Sf-def}
S^f(x_1,x_2):=\left\{\begin{split}
S(x_1,x_2)\quad\mbox{for}&\quad f(x_1)\le x_2< 1,\\
{S}^e\circ\mathfrak{P}_f(x_1,x_2)\quad\mbox{for}&\quad -\frac{1}{2}< x_2<f(x_1).
\end{split}\right. 
\end{equation}
Then 
\begin{equation*}
\|S^f-S_0^+\|_{2,\alpha,\mathcal{N}_{L,-1/2}^+}\le C^{\ast\ast}\|S-S_0^+\|_{2,\alpha,\mathcal{N}_{L,f}^+}
\end{equation*}
for a constant $C^{\ast\ast}>0$ depending only on the data and $\alpha$ but independent of $L$.
By Lemma \ref{2D-trans}, we have the estimate 
\begin{equation}\label{Sf-est}
\|S^f-S_0^+\|_{2,\alpha,\mathcal{N}_{L,-1/2}^+}\le C^{\ast\ast}\|S-S_0^+\|_{2,\alpha,\mathcal{N}_{L,f}^+}\le C^{\ast\ast}C^{\ast}\|S_{\rm en}-S_0\|_{2,\alpha,\Gamma_{\rm en}^+}
\end{equation}
with a constant $C^{\ast}>0$ in \eqref{2D-def-C}.

By a direct computation, one can easily check that there exists a constant $\epsilon_4>0$ depending only on the data so that if 
$$M_1\sigma+M_2\sigma+M_3\sigma+M_4\sigma\le\epsilon_4,$$
 then we have 
\begin{equation*}
\|H(S_{\ast},\nabla\varphi,\nabla\psi)(\nabla\varphi+\nabla^{\bot}\psi)-(\rho_0^+u_0,0)\|_{1,\alpha,\mathcal{N}_{L,f}^+}\le C^{\ast\ast\ast}\left(M_1+1\right)\sigma
\end{equation*}
for a constant $C^{\ast\ast\ast}>0$ depending only on the data and $\alpha$ but independent of $L$.
If it holds that 
$$\sigma\le \frac{\rho_0^+u_0}{2C^{\ast\ast\ast}\left(M_1+1\right)},$$
then 
\begin{equation}\label{S-comp-1}
\|H(S_{\ast},\nabla\varphi,\nabla\psi)(\nabla\varphi+\nabla^{\bot}\psi)-(\rho_0^+u_0,0)\|_{1,\alpha,\mathcal{N}_{L,f}^+}\le \frac{\rho_0^+u_0}{2}.
\end{equation}
Also, by the boundary conditions \eqref{S-free} for $(\varphi,\psi)$ and the definition of $\varphi_{\rm en}$ given in \eqref{def-phi-en}, we have 
\begin{equation}\label{S-comp-2}
\partial_{x_2}\varphi+\partial_{x_1}\psi\equiv 0\quad\mbox{on}\quad\Gamma_{\rm en}^{\epsilon}\cup\Gamma_{\rm ex}^{L,f}.
\end{equation}
It follows from \eqref{S-comp-1}, \eqref{S-comp-2}, and the equation in \eqref{Ite-3D2} that  
$$\partial_{x_1}S^f\equiv 0\quad\mbox{on}\quad\Gamma_{\rm  en }^{\epsilon}\cup\Gamma_{\rm ex}^L.$$

We define an iteration mapping $\mathcal{J}:\mathcal{S}(M_1)\rightarrow C^{2,\alpha/2}(\overline{\mathcal{N}_{L}})$ by 
$$\mathcal{J}(S_{\ast})=S^f$$
for $S^f$ given by \eqref{Sf-def}.
Choose $M_1$ and $\sigma_4^{\ast}$ as
\begin{equation}\label{sigma4star}
M_1=C^{\ast\ast}C^{\ast}\,\mbox{ and }\,\sigma_4^{\ast}=\min\left\{\sigma_4^{\ast\ast},\frac{\epsilon_4}{M_1+M_2+M_3+M_4},\frac{\rho_0^+u_0}{2C^{\ast\ast\ast}\left(M_1+1\right)}\right\}
\end{equation}
for a constant $C^{\ast\ast}C^{\ast}>0$ in \eqref{Sf-est} and $\sigma_4^{\ast\ast}$ in Lemma \ref{2D-trans}.
Then the mapping $\mathcal{J}$ maps $\mathcal{S}(M_1)$ into itself if $\sigma\le\sigma_4^{\ast}$.

The iteration set $\mathcal{S}(M_1)$ given by \eqref{S-set} is a convex and compact subset of $C^{2,\alpha/2}(\overline{\mathcal{N}_{L,-1/2}^+})$. 
Suppose that a sequence $\{S_{\ast}^{(k)}\}_{k=1}^{\infty}\subset\mathcal{S}(M_1)$ converges in $C^{2,\alpha/2}(\overline{\mathcal{N}_{L,-1/2}^+})$ to $S_{\ast}^{(\infty)}\in\mathcal{S}(M_1)$.
For each $k\in\mathbb{N}\cup\{\infty\}$, set 
\begin{equation*}
S^{(k)}:=\mathcal{J}(S_{\ast}^{(k)}).
\end{equation*}
And, let $(f^{(k)},\varphi^{(k)},\psi^{(k)})\in C^{2,\alpha}([0,L])\times\left[C^{2,\alpha}(\overline{\mathcal{N}_{L,f^{(k)}}^+})\right]^2$ be the unique solution of Problem \ref{FBP-S} associated with $S_{\ast}=S_{\ast}^{(k)}$. 
By the uniqueness of a solution for Problem \ref{FBP-S}, $\{f^{(k)}\}_{k=1}^{\infty}$ is convergent in $C^{2,\alpha/2}([0,L])$. 
Denote its limit by $f^{(\infty)}$ and the unique solution to \eqref{Ite-3D} associated with $f_\ast=f^{(\infty)}$ and $S_{\ast}=S_{\ast}^{(\infty)}$ by $(\varphi^{(\infty)},\psi^{(\infty)})$.
Set 
\begin{equation*}
{\bf V}^{(k)}:=H\left(S^{(k)}_{\ast},\nabla\varphi^{(k)},\nabla\psi^{(k)}\right)\left(\nabla\varphi^{(k)}+\nabla^{\bot}\psi^{(k)}\right),
\end{equation*}
and define a transformation $T^{(k)}:\overline{\mathcal{N}_{L,f^{(\infty)}}^+}\rightarrow\overline{\mathcal{N}_{L,f^{(k)}}^+}$ by 
\begin{equation*}
T^{(k)}(x_1,x_2)=\left(x_1,\frac{1-f^{(k)}(x_1)}{1-f^{(\infty)}(x_1)}(x_2-1)+1\right).
\end{equation*}
Then ${\bf V}^{(k)}\circ T^{(k)}$ converges to 
\begin{equation*}
{\bf V}^{(\infty)}=H\left(S^{(\infty)}_{\ast},\nabla\varphi^{(\infty)},\nabla\psi^{(\infty)}\right)\left(\nabla\varphi^{(\infty)}+\nabla^{\bot}\psi^{(\infty)}\right)
\end{equation*}
in $C^{1,\alpha/2}(\overline{\mathcal{N}_{L,f^{(\infty)}}^+})$.
By Lemma \ref{2D-trans}, 
$S^{(k)}$ converges to $S^{(\infty)}$ in $C^{2,\alpha/2}(\overline{\mathcal{N}_{L,-1/2}^+})$.
This implies that $\mathcal{J}(S_{\ast}^{(k)})$ converges to $\mathcal{J}(S_{\ast}^{(\infty)})$ in $C^{2,\alpha/2}(\overline{\mathcal{N}_{L,-1/2}^+})$.
Thus $\mathcal{J}$ is a continuous map in $C^{2,\alpha/2}(\overline{\mathcal{N}_{L,-1/2}^+})$. 
Applying the Schauder fixed point theorem yields that $\mathcal{J}$ has a fixed point in $\mathcal{S}(M_1)$, say $S^f$. 
For such a fixed point $S^f$, let $(f,\varphi,\psi)\in C^{2,\alpha}([0,L])\times\left[C^{2,\alpha}(\overline{\mathcal{N}_{L,f}^+})\right]^2$ be the unique solution of Problem \ref{FBP-S}.
Then $(f,S,\varphi,\psi)$ solves the Problem \ref{Prob-Cut}  provided that $\sigma\le\sigma_4^{\ast}$.

Let 
$(f^{(1)},S^{(1)},\varphi^{(1)},\psi^{(1)})$ and  $(f^{(2)},S^{(2)},\varphi^{(2)},\psi^{(2)})$ be two solutions to Problem \ref{Prob-Cut}, and suppose that each solution satisfies the estimate given in \eqref{3D-Prop-est} of Proposition \ref{3D-Prop}.
For a transformation $\mathfrak{T}:\overline{\mathcal{N}_{L,f^{(1)}}^+}\rightarrow\overline{\mathcal{N}_{L,f^{(2)}}^+}$ defined by \eqref{Trans-12}, set 
\begin{equation*}
\left\{\begin{split}
& \widetilde{f}:=f^{(1)}-f^{(2)},\quad\widetilde{S}:=S^{(1)}-\left(S^{(2)}\circ\mathfrak{T}\right),\\
&\widetilde{\phi}:=\varphi^{(1)}-\left(\varphi^{(2)}\circ\mathfrak{T}\right),\quad\widetilde{\psi}:=\psi^{(1)}-\left(\psi^{(2)}\circ\mathfrak{T}\right),\\
\end{split}\right.
\end{equation*}
By a direct computation, one can easily check that there exists a constant $C_1^{\star}>0$ depending only on the data and $\alpha$ but independent of $L$ so that 
if $\sigma\le C_1^{\star}$, then 
\begin{equation}\label{S-estimate}
\begin{split}
\|\widetilde{S}\|_{1,\alpha,\mathcal{N}_{L,f^{(1)}}^+}
&\le C\sigma\left(\|\widetilde{\phi}\|_{2,\alpha,\mathcal{N}_{L,f^{(1)}}^+}+\|\widetilde{\psi}\|_{2,\alpha,\mathcal{N}_{L,f^{(1)}}^+}+\|\widetilde{f}\|_{2,\alpha,(0,L)}\right)\\
&\le C\sigma\|\widetilde{f}\|_{2,\alpha,(0,L)},
\end{split}
\end{equation}
where the constant $C>0$ depends only on the data and $\alpha$ but independent of $L$.
By adjusting the proof of Lemma \ref{S-free-2} with using the estimate \eqref{S-estimate}, we have the estimate
\begin{equation}\label{f-est-SS}
\|\widetilde{f}\|_{2,\alpha,(0,L)}\le C_2^{\star}\sigma\|\widetilde{f}\|_{2,\alpha,(0,L)}
\end{equation}
for a constant $C_2^{\star}>0$ depending only on the data and $\alpha$ but independent of $L$.
Choose $\sigma_4$ as
$$\sigma_4=\min\left\{\sigma_4^{\ast}, C_1^{\star}, \frac{1}{2C_2^{\star}}\right\}$$
for $\sigma_4^{\ast}$ defined in \eqref{sigma4star}.
Then we obtain from \eqref{f-est-SS} that $f^{(1)}=f^{(2)}$.
By \eqref{S-estimate}, we have $S^{(1)}=S^{(2)}$.
Therefore $(f^{(1)},S^{(1)},\varphi^{(1)},\psi^{(1)})=(f^{(2)},S^{(2)},\varphi^{(2)},\psi^{(2)})$ by Lemma \ref{S-free-2}.
The proof of Proposition \ref{3D-Prop} is completed.
\qed

\section{Free boundary problem in the infinitely long nozzle $\mathcal{N}$}\label{sec-ex}
\subsection{Proof of Theorem \ref{3D-HD}}
\label{subsec-ex}
Let $\sigma_4$ be from Proposition \ref{3D-Prop} and suppose that $\sigma\le\sigma_4$.
By Proposition \ref{3D-Prop}, Problem \ref{Prob-Cut} has a solution for each $L>0$.
For each $m\in\mathbb{N}$, let $(f^{(m)}, S^{(m)}, \varphi^{(m)},\psi^{(m)})$ be a solution of Problem \ref{Prob-Cut} in $\mathcal{N}_{m+20}:=\mathcal{N}\cap\{0<x_1<m+20\}$, and suppose that the solution satisfies the estimate \eqref{3D-Prop-est} given in Proposition \ref{3D-Prop}. 
Then, using the Arzel\'a-Ascoli theorem and a diagonal procedure, we can extract a subsequence, still written as $\{\left(f^{(m)}, S^{(m)},\varphi^{(m)},\psi^{(m)}\right)\}_{m\in\mathbb{N}}$ so that the subsequence converges to functions $(f^{\ast},S^{\ast},\varphi^{\ast},\psi^{\ast})$ in the following sense:
for any $L>0$,
\begin{itemize}
\item[(i)] $f^{(m)}$ converges to $f^{\ast}$ in $C^2$ in $[0,L]$;
\item[(ii)] $(S^{(m)}\circ T^{(m)},\varphi^{(m)}\circ T^{(m)},\psi^{(m)}\circ T^{(m)})$ converges to $(S^{\ast},\varphi^{\ast},\psi^{\ast})$ in $C^2$ in $\overline{\mathcal{N}_{L,f^{\ast}}^+}$,
where $T^{(m)}:  \overline{\mathcal{N}_{m+20,f^{\ast}}^+}\rightarrow\overline{\mathcal{N}_{m+20,f^{(m)}}^+}$ is defined by 
\begin{equation*}
{T}^{(m)}(x_1,x_2)=\left(x_1,\frac{1-f^{(m)}(x_1)}{1-f^{\ast}(x_1)}x_2+\frac{f^{(m)}(x_1)-f^{\ast}(x_1)}{1-f^{\ast}(x_1)}\right).
\end{equation*}
\end{itemize}
By a change of variables and passing to the limit $m\rightarrow\infty$, we can prove that 
$(f^{\ast},S^{\ast},\varphi^{\ast},\psi^{\ast})$ is a solution to the free boundary problem \eqref{H-E-S} with boundary conditions \eqref{Free-BC} and \eqref{2D_BC_2}. 
Furthermore, it follows from the $C^2$ convergence of $\{(f^{(m)},S^{(m)},\varphi^{(m)},\psi^{(m)})\}_{m\in\mathbb{N}}$ and the estimate \eqref{3D-Prop-est} given in Proposition \ref{3D-Prop} that $(f^{\ast},S^{\ast},\varphi^{\ast},\psi^{\ast})$ satisfy the estimate \eqref{Th4-est} for a constant $C>0$ depending only on the data and $\alpha$.
\qed

\subsection{Proof of Theorem \ref{3D-Far-Thm} (b)}\label{sec-far}
 To compute the asymptotic limit of the solution to Problem \ref{3D-Prob} at $x_1=\infty$, we follow the idea from \cite{Ch-D-X, X-X}.

Let $\sigma_1$ be from Theorem  \ref{3D-inf-Thm} (a).
By Theorem \ref{3D-inf-Thm} (a), if $\sigma\le\sigma_1$, then there exists a solution $(g_D,u_1,u_2,\rho,p)$ of Problem \ref{3D-Prob} satisfying the estimate \eqref{Th3-est}.

By the continuity equation $\mbox{div}(\rho {\bf u})=0$, the function $\mathfrak{h}$ given by 
\begin{equation*}
\mathfrak{h}(x_1,x_2):=\int_{1}^{x_2}\rho u_1(x_1,t)dt\quad\mbox{for}\quad (x_1,x_2)\in\overline{\mathcal{N}_{g_D}^+}
\end{equation*}
satisfies 
\begin{equation}\label{2D-st-xr}
\partial_{x_1}\mathfrak{h}=-\rho u_2,\quad\partial_{x_2}\mathfrak{h}=\rho u_1.
\end{equation}
By \eqref{2D-S-EN}-\eqref{mathcalG}, the entropy $S(=p/\rho^{\gamma})$ is represented as 
$$S({\rm x})=S_{\rm en}\circ\mathcal{Y}_0({\rm x})=S_{\rm en}\circ\mathcal{G}^{-1}(\mathfrak{h}({\rm x}))=:\mathcal{S}(\mathfrak{h}({\rm x})),\quad{\rm x}\in\overline{\mathcal{N}_{g_D}^+},$$
where $\mathcal{Y}_0$ and $\mathcal{G}$ are given by \eqref{y0-def} and \eqref{mathcalG}.
Since $S_{\rm en}$ and $\mathcal{G}^{-1}$ are differentiable, $\mathcal{S}$ is a differentiable function of $\mathfrak{h}$.
Set
$$\mathfrak{S}(\mathfrak{h}):=\frac{\gamma}{\gamma-1}\mathcal{S}(\mathfrak{h}).$$
Then, by the definition of the Bernoulli invariant \eqref{Ber}, we have 
\begin{equation}\label{2D-BER}
B_0^+\rho^2=\frac{1}{2}|\nabla\mathfrak{h}|^2+\mathfrak{S}(\mathfrak{h})\rho^{\gamma+1}\quad\mbox{in}\quad\overline{\mathcal{N}_{g_D}^+}.
\end{equation}
By differentiating the equation \eqref{2D-BER} with respect to $x_1$ and $x_2$, we have 
\begin{equation}\label{2D-partial-rho}
\begin{split}
&\partial_{x_1}\rho=-\frac{(\partial_{x_1}\mathfrak{h})(\partial_{x_1x_1}\mathfrak{h})+(\partial_{x_2}\mathfrak{h})(\partial_{x_1x_2}\mathfrak{h})+(\partial_{x_1}\mathfrak{h})\mathfrak{S}'\rho^{\gamma+1}}{(\gamma+1)\mathfrak{S}\rho^{\gamma}-2B_0^+\rho},\\
&\partial_{x_2}\rho=-\frac{(\partial_{x_1}\mathfrak{h})(\partial_{x_1x_2}\mathfrak{h})+(\partial_{x_2}\mathfrak{h})(\partial_{x_2x_2}\mathfrak{h})+(\partial_{x_2}\mathfrak{h})\mathfrak{S}'\rho^{\gamma+1}}{(\gamma+1)\mathfrak{S}\rho^{\gamma}-2B_0^+\rho},
\end{split}
\end{equation}
where $'$ denotes the derivative with respect to $\mathfrak{h}$.
Using \eqref{2D-st-xr}-\eqref{2D-partial-rho}, the equation 
\begin{equation}\label{ES-eq}
\partial_{x_1}(\rho u_1u_2)+\partial_{x_2}(\rho u_2^2)+\partial_{x_2} p=0\quad\mbox{in}\quad\mathcal{N}_{g_D}^+
\end{equation}
in \eqref{E-S} can be rewritten as 
\begin{equation}\label{2D-infty-div}
\nabla\cdot\left(\frac{\nabla\mathfrak{h}}{\rho}\right)=-\frac{1}{\gamma}\mathfrak{S}'\rho^{\gamma}\quad\mbox{in}\quad\mathcal{N}_{g_D}^+.
\end{equation}
Set 
$$\omega:=\partial_{x_1}\mathfrak{h},$$
and differentiate \eqref{2D-infty-div} with respect to $x_1$ to get the following equation for $\omega$:
\begin{equation}\label{2D-infty-omega}
\partial_i\left(\frac{\mathfrak{q}_{ij}}{\rho^2}\partial_j\omega\right)-\partial_i\left(\frac{\mathfrak{t}_2(\partial_i\mathfrak{h})}{\rho^2}\omega\right)
=\mathfrak{q}_1\omega+\mathfrak{q}_2(\partial_i\mathfrak{h})(\partial_i\omega)\quad\mbox{in}\quad\mathcal{N}_{g_D}^+,
\end{equation}
where $\mathfrak{q}_{ij}$, $\mathfrak{q}_1$, and $\mathfrak{q}_2$ are defined by 
\begin{equation*}\label{2D-qqq}
\begin{split}
&\mathfrak{q}_{ij}:=\rho\delta_{ij}-2\mathfrak{t}_1(\partial_i\mathfrak{h})(\partial_j\mathfrak{h}),\\
&\mathfrak{q}_1:=-\frac{1}{\gamma}\mathfrak{S}''(\mathfrak{h})\rho^{\gamma}-\mathfrak{S}'(\mathfrak{h})\rho^{\gamma-1}\mathfrak{t}_2,\\
&\mathfrak{q}_2:=-2\mathfrak{S}'(\mathfrak{h})\rho^{\gamma-1}\mathfrak{t}_1.
\end{split}
\end{equation*}
Here, $\mathfrak{t}_1$ and $\mathfrak{t}_2$ are given by 
\begin{eqnarray}
&\label{rho1}&\mathfrak{t}_1:=2\left(\frac{\partial\rho}{\partial|\nabla\mathfrak{h}|^2}\right)=-\frac{1}{\rho(c^2-|\nabla\mathfrak{h}|^2/\rho^2)},\\
&\label{rho2}&\mathfrak{t}_2:=\frac{\partial\rho}{\partial \mathfrak{h}}=-\frac{\mathfrak{S}'\rho^{\gamma+1}}{\rho(c^2-|\nabla\mathfrak{h}|^2/\rho^2)},
\end{eqnarray}
for $c=(\gamma-1)\mathfrak{S}\rho^{\gamma-1}$. 
Since $\varphi,\psi\in C^{3,\alpha}(\mathcal{N}_{g_D}^+)$ for ${\bf u}=\nabla\varphi+\nabla^{\bot}\psi$ by the standard elliptic theory \cite[Theorem 6.17]{G-T}, 
$\mathfrak{h}\in C^{3,\alpha}(\mathcal{N}_{g_D}^+)$ and the equation \eqref{2D-infty-omega} is well-defined.

By the boundary conditions \eqref{3D-BC-ent} and \eqref{2D-slip}, $\omega$ satisfies
\begin{equation*}\label{BC-omega}
\left.\begin{split}
\omega=-\rho v_{\rm en}\quad\mbox{on}\quad\Gamma_{\rm en}^+,\quad \omega=0\quad\mbox{on}\quad\Gamma_{\rm w}^+.\\
\end{split}\right.
\end{equation*}
In order to get a conormal boundary condition for \eqref{2D-infty-omega} on $\Gamma_{g_D}$, we compute
$\left(\frac{\mathfrak{q}_{1j}}{\rho^2}\partial_j\omega,\frac{\mathfrak{q}_{2j}}{\rho^2}\partial_j\omega\right)\cdot{\bf n}_{g_D}$ to get
\begin{equation}\label{co-boundary}
\left.\begin{split}
\left(\frac{\mathfrak{q}_{1j}}{\rho^2}\partial_j\omega,\frac{\mathfrak{q}_{2j}}{\rho^2}\partial_j\omega\right)\cdot{\bf n}_{g_D}=\widetilde{\mu}\omega\quad&\mbox{on}\quad\Gamma_{g_D}=\{x_2=g_D(x_1),\,\,x_1>0\}
\end{split}\right.
\end{equation}
for  $\widetilde{\mu}$ defined by 
\begin{equation}\label{2D-omega-BC}
\begin{split}
\widetilde{\mu}:=&-\frac{\mathfrak{q}_{11}\partial_{x_1x_1}\mathfrak{h}+\mathfrak{q}_{12}\partial_{x_1x_2}\mathfrak{h}-2\mathfrak{t}_1(\partial_{x_2}\mathfrak{h})^2\partial_{x_1x_1}\mathfrak{h}}{\rho^2(\partial_{x_2}\mathfrak{h})\sqrt{1+|g_D'|^2}}\\
&+\frac{\mathfrak{q}_{22}}{\rho^2\sqrt{1+|g_D'|^2}}\left\{\frac{\partial_{x_1x_1}\mathfrak{h}+\partial_{x_2x_2}\mathfrak{h}}{(\partial_{x_1}\mathfrak{h})g_D'+\partial_{x_2}\mathfrak{h}}\right\},
\end{split}
\end{equation}
where we represent ${\bf n}_{g_D}$ as 
$${\bf n}_{g_D}=-\frac{1}{\sqrt{1+|g_D'|^2}}\left(\frac{\omega}{\partial_{x_2}\mathfrak{h}},1\right).$$
By the definition of Bernoulli invariant \eqref{Ber},
\begin{equation}\label{2D-fxy-cc}
(\partial_{x_1}\mathfrak{h})^2+(\partial_{x_2}\mathfrak{h})^2=\mathcal{C}\quad\mbox{on}\quad\Gamma_{g_D}
\end{equation}
for $
\mathcal{C}:=2\left(B_0^+p_0^{2/\gamma}S_{\rm en}^{-2/\gamma}(0)-\frac{\gamma}{\gamma-1}p_0^{1+1/\gamma}S_{\rm en}^{-1/\gamma}(0)\right).
$
Differentiating the equation \eqref{2D-fxy-cc} in the tangential direction along $\Gamma_{g_D}$, we have 
\begin{equation*}
(\partial_{x_1}\mathfrak{h})(\partial_{x_1x_1}\mathfrak{h}+g_D'(x_1)\partial_{x_1x_2}\mathfrak{h})+(\partial_{x_2}\mathfrak{h})(\partial_{x_2x_1}\mathfrak{h}+g_D'(x_1)\partial_{x_2x_2}\mathfrak{h})=0\quad\mbox{on}\quad\Gamma_{g_D},
\end{equation*}
which provides
\begin{equation}\label{2D-fxy-xy}
\partial_{x_1x_2}\mathfrak{h}
=-\frac{\partial_{x_1x_1}\mathfrak{h}+\partial_{x_2x_2}\mathfrak{h}}{g_D'(x_1)\partial_{x_1}\mathfrak{h}+\partial_{x_2}\mathfrak{h}}\omega\quad\mbox{on}\quad\Gamma_{g_D}.
\end{equation}
By straightforward computations with using \eqref{2D-fxy-xy}, we get the conormal boundary condition \eqref{co-boundary}.

Fix a constant $L>0$ and let $\eta$ be a $C^{\infty}$ function satisfying 
\begin{equation*}\label{2D-eta}
\eta=1\quad\mbox{for}\quad 0\le x_1<L,\quad\eta=0\quad\mbox{for}\quad x_1>L+1,\quad |\eta'(x_1)|\le 2.
\end{equation*}
Multiply \eqref{2D-infty-omega} by  $\eta^2\omega$, and integrate the result over the domain $N_{g_D}^+$ to get
\begin{equation}\label{2D-sum-omega}
\iint_{\mathcal{N}^+_{g_D}}\frac{\eta^2|\nabla\omega|^2}{\rho}dx_2dx_1=\sum_{i=1}^6 I_i+\sum_{i=1}^2B_i,
\end{equation}
for  
\begin{equation*}
\begin{split}
&I_1:=-\iint_{\mathcal{N}^+_{g_D}}\frac{|\nabla\mathfrak{h}\cdot\nabla\omega|^2\eta^2}{\rho(\rho^2c^2-|\nabla\mathfrak{h}|^2)}dx_2dx_1,\\
&I_2:=-2\iint_{\mathcal{N}^+_{g_D}}\frac{\mathfrak{q}_{ij}}{\rho^2}\eta\omega(\partial_j\omega)(\partial_i\eta)dx_2dx_1,\\
&I_3:=2\iint_{\mathcal{N}^+_{g_D}}\frac{\mathfrak{t}_2\nabla\mathfrak{h}\cdot\nabla\eta}{\rho^2}\eta\omega^2dx_2dx_1,\\
&I_4:=-2\iint_{\mathcal{N}^+_{g_D}}\frac{\mathfrak{S}'\rho^{\gamma}}{\rho^2c^2-|\nabla\mathfrak{h}|^2}(\partial_i\mathfrak{h})(\partial_i\omega)\eta^2\omega dx_2dx_1,\\
&I_5:=\iint_{\mathcal{N}^+_{g_D}}\frac{1}{\gamma}\mathfrak{S}''\rho^{\gamma}\eta^2\omega^2dx_2dx_1,\\
&I_6:=-\iint_{\mathcal{N}^+_{g_D}}\frac{\rho^2(\mathfrak{S}'\rho^{\gamma-1})^2}{\rho^2c^2-|\nabla\mathfrak{h}|^2}\rho\eta^2\omega^2dx_2dx_1,\\
&B_1:=\int_{\Gamma_{g_D}\cup\Gamma_{\rm en}^+}\left(\frac{\mathfrak{q}_{ij}}{\rho^2}\partial_j \omega\right)\eta^2\omega\cdot{\bf n}_{\rm out} ds,\\
&B_2:=-\int_{\Gamma_{g_D}\cup\Gamma_{\rm en}^+}\left(\frac{\mathfrak{t}_2\partial_i\mathfrak{h}}{\rho^2}\right)\eta^2\omega^2\cdot{\bf n}_{\rm out} ds.
\end{split}
\end{equation*}
We claim that 
\begin{equation}\label{2D-IJ-est}
\left\{\begin{split}
&I_1+I_4+I_6\le 0,\\
&|I_k|\le C\int_L^{L+1}\int_{g_D(x_1)}^1 |\nabla \omega|^2 dx_2dx_1\quad\mbox{for}\quad k=2,3,\\
&|I_5|\le C{\sigma}\int_0^{L+1}\int_{g_D(x_1)}^1 |\nabla \omega|^2 dx_2dx_1,\\
&|B_k|\le C{\sigma}\int_0^{L+1}\int_{g_D(x_1)}^1|\nabla\omega|^2 dx_2dx_1+E\quad\mbox{for}\quad k=1,2, \\
\end{split}\right.
\end{equation}
where $E\ge0$ and $C>0$ are constants depending only on the data and $\alpha$.
From now on, the constant $C$ depends only on the data and $\alpha$, which may vary from line to line.

First, by the H\"older inequality, it holds that 
\begin{equation*}
\begin{split}
I_4&\le 2\left(\iint_{\mathcal{N}^+_{g_D}}\frac{|\nabla\mathfrak{h}\cdot\nabla\omega|^2\eta^2}{\rho(\rho^2c^2-|\nabla\mathfrak{h}|^2)}dx_2dx_1\right)^{1/2}\left(\iint_{\mathcal{N}^+_{g_D}}\frac{(\mathfrak{S}'\rho^{\gamma})^2\rho\eta^2\omega^2}{\rho^2c^2-|\nabla\mathfrak{h}|^2} dx_2dx_1\right)^{1/2}\\
&=2\sqrt{|I_1||I_6|}.
\end{split}
\end{equation*}
From this, we have 
$$I_1+I_4+I_6\le -|I_1|+2\sqrt{|I_1||I_6|}-|I_6|\le 0.$$

By a direct computation, one can easily check that there exists a constant $\sigma_{\star}\in (0,\sigma_1]$ depending only on the data and $\alpha$ so that if $\sigma\le \sigma_{\star}$, then we have 
\begin{equation}\label{far-est}
|\rho-\rho_0^+|\le \frac{\rho_0^+}{2},\quad  
|\mathfrak{O}_s-\mathfrak{O}_0|\le \frac{\mathfrak{O}_0}{2},\quad 
|\mathfrak{S}'(\mathfrak{h})|\le C\sigma,\quad\mbox{and}\quad
|\mathfrak{S}''(\mathfrak{h})|\le C{\sigma}
\end{equation}
in $\mathcal{N}_{g_D}^+$, where 
\begin{equation*}
\mathfrak{O}_0:=c_0^2-|{\bf u}_0|^2=\frac{\gamma p_0}{\rho_0^+}-u_0^2\quad\mbox{and}\quad\mathfrak{O}_s:=c^2-|{\bf u}|^2
=\frac{\gamma p}{\rho}-|{\bf u}|^2.
\end{equation*}
From \eqref{far-est}, we have 
\begin{equation}\label{rho12-mu}
|\mathfrak{t}_1|\le C,\quad |\mathfrak{t}_2|\le C|\mathfrak{S}'(\mathfrak{h})|\le C\sigma,\quad\mbox{and}\quad |\widetilde{\mu}|\le C\sigma\quad\mbox{in}\quad \mathcal{N}_{g_D}^+
\end{equation}
for $\mathfrak{t}_1$, $\mathfrak{t}_2$, and $\widetilde{\mu}$ defined by \eqref{rho1}, \eqref{rho2}, and \eqref{2D-omega-BC}, respectively.

Since $-2\omega(\partial_j\omega)\le \omega^2+|\nabla\omega|^2$ and $\rho\ge\rho_0^+/2$ in $\mathcal{N}_{g_D}^+$, we have 
\begin{equation}\label{2D-I2-est}
\begin{split}
&|I_2|\le C\int_L^{L+1}\int^1_{g_D(x_1)}(|\nabla\omega|^2+\omega^2)\,dx_2dx_1.\\
\end{split}
\end{equation}
By the slip boundary condition \eqref{2D-slip}, $\omega\equiv0$ on $\Gamma_{\rm w}^+$, and this yields that
\begin{equation*}\label{2D-omega-cont}
\omega(x_1,t)=-\int_{t}^1\partial_{x_2}\omega(x_1,x_2)\,dx_2\quad\mbox{for}\quad (x_1,t)\in\overline{\mathcal{N}_{g_D}^+}.
\end{equation*}
By the H\"older inequality,  we have
\begin{equation}\label{2D-omega22}
\omega^2(x_1,t)\le (1-t)\int_{t}^1(\partial_{x_2} \omega)^2(x_1,x_2)\,dx_2\le C\int_{g_D(x_1)}^1|\nabla\omega|^2\,dx_2
\end{equation}
for $(x_1,t)\in\overline{\mathcal{N}_{g_D}^+}$. 
Substituting \eqref{2D-omega22} into  \eqref{2D-I2-est}, we obtain
\begin{equation*}
\begin{split}
&|I_2|\le C\int_L^{L+1}\int^1_{g_D(x_1)}|\nabla\omega|^2\,dx_2dx_1.\\
\end{split}
\end{equation*}
By \eqref{far-est}-\eqref{rho12-mu} and \eqref{2D-omega22}, we have 
\begin{equation*}\label{2D-I35}
\begin{split}
|I_3|&\le C\int_L^{L+1}\int_{g_D(x_1)}^1\omega^2 \,dx_2dx_1\le C\int_L^{L+1}\int_{g_D(x_1)}^1|\nabla\omega|^2\, dx_2dx_1,\\
|B_2|
&\le C{\sigma}\int_0^{L+1}\int_{g_D(x_1)}^1|\nabla\omega|^2\, dx_2dx_1+E_2,
\end{split}
\end{equation*}
where $E_2\ge0$ is a constant depending only on the data and $\alpha$.
By \eqref{far-est} and \eqref{2D-omega22}, we have
\begin{equation*}
|I_5|\le C|\mathfrak{S}''|\int_0^{L+1}\int_{g_D(x_1)}^1\omega^2\,dx_2dx_1\le C{\sigma}\int_0^{L+1}\int_{g_D(x_1)}^1|\nabla\omega|^2\,dx_2dx_1.
\end{equation*}
It follows from \eqref{2D-omega-BC}, \eqref{far-est}, \eqref{rho12-mu}, and \eqref{2D-omega22} that 
\begin{equation*}
\begin{split}
|B_1|
\le C{\sigma}\int_0^{L+1}\int_{g_D(x_1)}^1|\nabla\omega|^2 \,dx_2dx_1+E_1,
\end{split}
\end{equation*}
where $E_1\ge0$ is a constant depending only on the data and $\alpha$.
Then the claim \eqref{2D-IJ-est} is proved.

From \eqref{2D-sum-omega}-\eqref{2D-IJ-est}, we have
\begin{equation*}
\begin{split}
\int_0^{L}\int_{g_D(x_1)}^1 |\nabla\omega|^2\, dx_2dx_1\le& C^{\sharp}{\sigma}\int_0^{L}\int_{g_D(x_1)}^1|\nabla\omega|^2\, dx_2dx_1\\
&+C\int_L^{L+1}\int_{g_D(x_1)}^1 |\nabla \omega|^2\, dx_2dx_1+CE,
\end{split}
\end{equation*}
where the constant $C^{\sharp}>0$ depends only on the data and $\alpha$.
If it holds that
$${\sigma}\le \frac{1}{2C^{\sharp}},$$
then we obtain from the previous estimate that  
\begin{equation*}\label{2D-RHS}
\int_0^{L}\int_{g_D(x_1)}^1 |\nabla\omega|^2\, dx_2dx_1\le C\int_L^{L+1}\int_{g_D(x_1)}^1 |\nabla \omega|^2 \,dx_2dx_1+CE.
\end{equation*}
Since $|\nabla\omega|\le C$ in $\overline{\mathcal{N}_{g_D}^+}$ by \eqref{Th3-est}, we have 
\begin{equation*}\label{2D-ven}
\int_0^L\int_{g_D(x_1)}^1|\nabla\omega|^2\, dx_2dx_1\le C
\end{equation*}
for some constant $C>0$ independent of $L$. 
Passing to the limit $L\rightarrow\infty$ yields 
\begin{equation*}
\int_0^{\infty}\int_{g_D(x_1)}^1|\nabla\omega|^2\, dx_2dx_1\le C.
\end{equation*}
Hence 
\begin{equation*}
\int_L^{L+1}\int_{g_D(x_1)}^1|\nabla\omega|^2\, dx_2dx_1\rightarrow 0\quad\mbox{as}\quad L\rightarrow\infty.
\end{equation*}
Since $\omega\in C^{1,\alpha}(\overline{\mathcal{N}^+_{g_D}})$, we have 
\begin{equation}\label{2D-omega-lim}
\|\nabla\omega(x_1,\cdot)\| _{C^0(\mathcal{N}_{g_D}^+\cap\{x_1>L\})}\rightarrow 0\quad\mbox{as}\quad L \rightarrow \infty.
\end{equation}
By \eqref{2D-omega-lim} and the slip boundary condition \eqref{2D-slip} on $\Gamma_{\rm w}^+$, we have 
\begin{equation}\label{omega0}
\|\omega(x_1,\cdot)\| _{C^0(\mathcal{N}_{g_D}^+\cap\{x_1>L\})}\rightarrow 0\quad\mbox{as}\quad L \rightarrow \infty.
\end{equation}
Since $\rho>\rho_0^+/2$ in $\mathcal{N}_{g_D}^+$ and $\omega=\partial_{x_1}\mathfrak{h}=-\rho u_2$, \eqref{omega0} implies that 
\begin{equation}\label{u2-0}
\|u_2(x_1,\cdot)\| _{C^0(\mathcal{N}_{g_D}^+\cap\{x_1>L\})}\rightarrow 0\quad\mbox{as}\quad L \rightarrow \infty.
\end{equation}
By \eqref{2D-omega-lim} and \eqref{u2-0}, we have 
\begin{equation}\label{nabla-u2}
\| u_2(x_1,\cdot)\| _{C^1(\mathcal{N}_{g_D}^+\cap\{x_1>L\})}\rightarrow 0\quad\mbox{as}\quad L \rightarrow \infty,
\end{equation}
from which 
\begin{equation*}
\|g_D'(x_1)\|_{C^1(\{x_1>L\})}\rightarrow 0\quad\mbox{as}\quad L\rightarrow\infty.
\end{equation*}
It follows from the equation \eqref{ES-eq} and \eqref{nabla-u2} that
\begin{equation}\label{p-0}
\|\partial_{x_2} p(x_1,\cdot)\| _{C^0(\mathcal{N}_{g_D}^+\cap\{x_1>L\})}\rightarrow 0\quad\mbox{as}\quad L\rightarrow\infty.
\end{equation}
By the boundary condition \eqref{3D-BC-Cont} and  \eqref{p-0}, we have 
\begin{equation*}\label{p-00}
\|p(x_1,\cdot)-p_0\| _{C^1(\mathcal{N}_{g_D}^+\cap\{x_1>L\})}\rightarrow 0\quad\mbox{as}\quad L\rightarrow\infty.
\end{equation*}
Furthermore, since we fixed $({\bf u}, \rho, p)$ in $\mathcal{N}\backslash\mathcal{N}_{g_D}^+$ to be
$({\bf u}, \rho, p)\equiv ({\bf 0}, \rho_0^-, p_0)$ right after Problem \ref{3D-Prob1} in Section \ref{sec-results}, we have 
\begin{equation*}
\begin{split}
&\lim_{L\to \infty}\| u_2(x_1,\cdot)\| _{C^1(\mathcal{N}\cap\{x_1>L\})} = 0,
\quad
\lim_{L\to \infty}\|p(x_1,\cdot)-p_0\| _{C^1(\mathcal{N}\cap\{x_1>L\})}= 0.
\end{split}
\end{equation*}
The proof of Theorem \ref{3D-Far-Thm} (b) is completed by choosing $\sigma_2$ as
$$\sigma_2=\min\left\{\sigma_1, \sigma_\star,\frac{1}{2C^{\sharp}}\right\}.$$
\qed


\vspace{.25in}
\noindent
{\bf Acknowledgements:}
The research of Myoungjean Bae was supported in part by  Samsung Science and Technology Foundation
under Project Number SSTF-BA1502-02.
The research of Hyangdong Park was supported in part by  Samsung Science and Technology Foundation
under Project Number SSTF-BA1502-02.

\end{document}